\documentclass[11pt]{article}
\usepackage[dvipsnames]{xcolor}
\usepackage{ulem}
\usepackage{pgfplots}
\usepackage{booktabs}
\usepackage{graphicx,wrapfig,lipsum}
\usepackage{float}    % For tables and other floats
\usepackage{verbatim} % For comments and other
\usepackage{amsmath, amssymb, amsthm, mathtools}  % For math
\usepackage{subfig}   % For subfigures
\usepackage[colorlinks=true,citecolor=blue,linkcolor=blue,urlcolor=blue]{hyperref}
\usepackage{bookmark}
\usepackage{fullpage}
\usepackage{enumerate}
\usepackage{paralist}
\usepackage{xspace}
\usepackage{multirow}
\usepackage{caption}
\usepackage{bbm}
\usepackage{bm}
\usepackage{url}
\usepackage{lipsum}
\usepackage{color}
\usepackage{microtype}
\usepackage[section]{placeins}
\usepackage{lscape}
\usepackage{multirow}
\usepackage{todonotes}

\usepackage{algorithm}
\usepackage{algorithmicx}
\usepackage{algpseudocode}
\usepackage{color}
\usepackage{multirow}
\usepackage{rotating}

\newcommand{\stoptocwriting}{%
	\addtocontents{toc}{\protect\setcounter{tocdepth}{-5}}}

\newcommand{\vertiii}[1]{{\left\vert\kern-0.25ex\left\vert\kern-0.25ex\left\vert #1
		\right\vert\kern-0.25ex\right\vert\kern-0.25ex\right\vert}}

\makeatletter

\newcommand{\vect}[1]{\ensuremath{\mathbf{#1}}}

\newcommand{\R}{\vect{R}}

\newcommand{\ccdot}{\boldsymbol{\cdot}}

\newcommand{\norm}[1]{\left\lVert#1\right\rVert}

\global\long\def\lone{1/1}
\global\long\def\loff{\mathrm{off}}

\newcommand{\thetatrue}{\theta^\star_t}
\newcommand{\thetaest}{\widehat{\theta}_t}
\newcommand{\bmest}{\widetilde{F}^*(\widehat{\mu}_t)}
\newcommand{\supp}{\mathrm{supp}}

\newcommand{\cD}{\mathcal{D}}

\newcommand{\cK}{\mathcal{K}}

\newcommand{\cO}{\mathcal{O}}

\newcommand{\cR}{\mathcal{R}}

\newcommand{\bE}{\mathbb{E}}

\newcommand{\bP}{\mathbb{P}}

\newcommand{\bR}{\mathbb{R}}

\global\long\def\linf{\infty/\infty}

\usepackage{adjustbox}

\newtheorem{theorem}{Theorem}
\newtheorem{definition}{Definition}
\newtheorem{lemma}{Lemma}

\newtheorem{remark}{Remark}

\newtheorem{proposition}{Proposition}

\newtheorem{assumption}{Assumption}

\usepackage{natbib}
\begin{document}

	\title{Solution Path of Time-varying Markov Random Fields with Discrete Regularization\thanks{Salar Fattahi is supported, in part, by grant 2152776 from the National Science Foundation, grant N00014-22-1-2127 from the Office of Naval Research, and a MICDE Catalyst Grant. Andr\'es G\'omez is supported, in part, by grant 2152777 from the National Science Foundation.}}
 \author{Salar Fattahi\\
 Industrial and Operations Engineering\\
	University of Michigan\\
 \href{mailto:fattahi@umich.edu}{fattahi@umich.edu}
 \and Andr\'es G\'omez\\
 Industrial and Systems Engineering\\
	University of Southern California\\
 \href{mailto:gomezand@usc.edu}{gomezand@usc.edu}
	}
	
	\maketitle
	\begin{abstract}
		 We study the problem of inferring sparse time-varying Markov random fields (MRFs) with different discrete and temporal regularizations on the parameters. Due to the intractability of discrete regularization, most approaches for solving this problem rely on the so-called \textit{maximum-likelihood estimation} (MLE) with relaxed regularization, which neither results in ideal statistical properties nor scale to the dimensions encountered in realistic settings. In this paper, we address these challenges by departing from the MLE paradigm and resorting to a new class of constrained optimization problems with exact, discrete regularization to promote sparsity in the estimated parameters. Despite the nonconvex and discrete nature of our formulation, we show that it can be solved efficiently and parametrically for \textit{all} sparsity levels. More specifically, we show that the entire solution path of the time-varying MRF for all sparsity levels can be obtained in $\mathcal{O}(pT^3)$, where $T$ is the number of time steps and $p$ is the number of unknown parameters at any given time. The efficient and parametric characterization of the solution path renders our approach highly suitable for cross-validation, where parameter estimation is required for varying regularization values. Despite its simplicity and efficiency, we show that our proposed approach achieves provably small estimation error for different classes of time-varying MRFs, namely Gaussian and discrete MRFs, with as few as one sample per time. Utilizing our algorithm, we can recover the complete solution path for instances of time-varying MRFs featuring over 30 million variables in less than 12 minutes on a standard laptop computer. Our code is available at \url{https://sites.google.com/usc.edu/gomez/data}.
  \footnote{An earlier version of this paper~\citep{fattahi2021scalable}, which concentrated solely on sparsely-changing Gaussian MRFs with a \textit{fixed} parameter, was presented at the Neural Information Processing Systems (NeurIPS) conference in 2021. The current submission is a substantial expansion of our work on multiple fronts. Firstly, we have shown that the complete solution path of time-varying MRFs can be recovered for all regularization parameters. Secondly, our new approach enables us to tackle a broader class of temporal changes beyond sparsity. Lastly, our statistical guarantees apply to a more extensive class of MRFs, specifically discrete MRFs.}
	\end{abstract}

	\stoptocwriting

 	\section{Introduction}
	Most modern systems are massive-scale with a hierarchy of unknown and ever-changing topologies that must be estimated in real-time. For instance, the real-time inference of time-varying
stock correlation networks is crucial for the identification of sharp changes in market conditions, and a need to
rebalance the portfolio~\citep{talih2005structural, polson2000bayesian, hallac2017network}. Another example is the inference of temporal gene regulatory networks based on gene expression data, which has immediate applications in early diagnosis of cancer and other disease processes~\citep{hartemink2000using, karlebach2008modelling, ravikumar2023efficient}. 

The behavior of time-varying networks, such as those mentioned above, can be captured by \textit{time-varying Markov Random Fields} (MRF). Time-varying MRFs are associated with a temporal sequence of undirected \textit{Markov graphs} {$\mathcal{G}_t(V,E_t)$, where $V$ is the set of nodes and $E_t$ is the set of edges in the graph at time $t$}. The node set $V$ represents the random variables in the model, while the edge set $E_t$ captures the conditional dependency between these variables that may change over time. In other words, if two nodes $i$ and $j$ are not connected by an edge, then the random variables $i$ and $j$ are independent at time $t$, conditioned on the rest of the variables; a feature known as \textit{Markov property}. 
The identification of such conditional independence properties would lead to simpler and more interpretable models, which are invaluable for improving human understanding of different phenomena. 

\subsection{Maximum Likelihood Estimation}
A popular approach for the inference of time-varying MRFs is based on the so-called \textit{maximum likelihood estimation} (MLE): to find a probabilistic graphical model, based on which the observed data is most probable to occur~\citep{wainwright2008graphical}. Despite being known as theoretically powerful tools---a fact noted as early as 1960s~\citep{kalman1960new, shellenbarger1966estimation}---MLE-based methods face several fundamental challenges. 

First, MLE-based approaches have overwhelmingly high computational complexity, which limits their application to small- and medium-scale problems. This challenge is further exacerbated in the time-varying setting where a graphical model needs to be inferred for {each time}, thereby leading to a dramatic increase in the number of parameters to be estimated. For instance, in order to obtain an $\epsilon$-accurate solution, typical MLE-based methods have complexity ranging from $\mathcal{O}(Tp^6\log(1/\epsilon))$ (via general interior-point methods)~\citep{mohan2014node, potra2000interior} to $\mathcal{O}(Tp^3/\epsilon)$ (via tailored first-order methods, such as ADMM)~\citep{hallac2017network, ravikumar2010high, banerjee2008model}. Here, $p$ is the number of unknown parameters, and $T$ is the number of time steps. Solvers with such computational complexity fall short of practical use in settings where the dimension $p$ is large, and the desired $\epsilon$ is small. This prohibitive complexity of MLE-based methods is also exemplified in their practical performance~\citep{fattahi2019graphical, zhang2018large, fattahi2019linear}.

Second, MLE-based methods fail to efficiently incorporate {prior} structural and temporal information into their estimation procedure. For instance, it is well-known that different time-varying networks exhibit sparse topologies, which can be captured via a combinatorial or {discrete regularizer} like the ``$\ell_0$-norm''. However, due to the combinatorial nature of the $\ell_0$-norm, MLE-based methods inevitably resort to {relaxed} or weaker variants of such regularization, such as $\ell_1$-norm, thereby suffering from inferior statistical properties~\citep{bertsimas2019certifiably}. In particular, a regularizer based on $\ell_1$-norm results in shrinkage of the entries of the precision matrix, and the bias due to the penalty can dominate the variance from the likelihood, ultimately leading to poor estimates. In \S2.2, we provide an example to illustrate the performance of such relaxation. 

Finally, in most applications, one needs to obtain a statistically meaningful regularization parameter via cross-validation~\citep{cox1979theoretical}, which amounts to solving the MLE repeatedly or parametrically for different values of the regularization parameter. However, obtaining parametric solutions to optimization or inference problems is a highly non-trivial task even if the fixed-parameter problem can be easily tackled. For example, while the well-known lasso regression can be solved efficiently, parametric algorithms such as LARS \citep{efron2004least} have in general exponential complexity \citep{mairal2012complexity}. 

\subsection{Our contributions}
To address the aforementioned challenges, our approach departs from the conventional wisdom in statistics and machine learning that the inference with discrete regularizers is intractable, and convex surrogates should be used instead. In fact, we prove the contrary: time-varying MRFs with discrete regularizers can be inferred {efficiently and parametrically}; that is, for \textit{all} sparsity  levels. Our contributions are as follows:
\begin{itemize}
\item[-] {\bf More tractable formulation:} We introduce a much simpler formulation for the inference of time-varying MRFs that can leverage and promote different sparsity and smoothness structures. Unlike most MLE-based methods which typically rely on convexity of the regularizer to ensure tractability, our proposed formulation can directly incorporate the nonconvex and discrete $\ell_0$ regularizer to promote sparsity in the estimated parameters and/or their differences.
\item [-] {\bf Recovering the solution path via dynamic programming:}  
We propose a dynamic programming approach that can recover the {entire} solution path of time-varying MRFs for \textit{all} sparsity levels, i.e., for all values of the regularization coefficients controlling the sparsity, with complexity $\cO(pT^3)$ (Theorem~\ref{thm_runtime}), which scales linearly with the dimension $p$. This makes our approach particularly suitable for cross-validation, where the goal is to find the regularization parameter that leads to the best generalization performance of the model.

\item [-] {\bf Statistical guarantees:} We show that our proposed estimation method enjoys a strong statistical guarantee, provided that an approximate backward mapping of the underlying distribution is available (Theorem~\ref{thm_constrained}). For two classes of Gaussian and Discrete time-varying MRFs, we derive sharp non-asymptotic guarantees for the estimated parameters, showing that they can be accurately inferred with as
few as \textit{one sample per time} (Propositions~\ref{prop_gmrf_sample_ker} and~\ref{prop_dmrf_sample_ker}). In the case where the underlying Markov graphs and their temporal differences are sparse, we establish high probability guarantees on the sparsistency of the obtained solutions and their differences. 

\item [-] {\bf Implementation:}
Using our algorithm, we can recover the entire solution path of an instance of time-varying MRFs with more than 30M variables in less than 12 minutes on a normal laptop computer. Our code, as well as our test cases, are available at \url{https://sites.google.com/usc.edu/gomez/data}. 
\end{itemize}

% \section{Related work}

\noindent{\bf Notations.} 
The $i^{th}$ element of a time-series vector $v_t$ is denoted as $v_{t;i}$; the $(i,j)^{th}$  element of a time-indexed matrix $V_t$ is denoted as $V_{t;ij}$. For a vector $v$, the notation $v_{i:j}$ is used to denote the subvector of $v$ from index $i$ to $j$. The $\ell_q$-ball with radius $\rho$ centered at $\bar\mu$ is defined as $\mathcal{B}_{q}(\bar\mu, \rho) := \{\mu : \|\mu-\bar\mu\|_q\leq \rho\}$.
For a vector $v$, the notations $\|v\|_\infty$, $\|v\|_2$, $\|v\|_0$ denote the $\ell_\infty$ norm, $\ell_2$ norm, and ``$\ell_0$-norm'', i.e., the number of nonzero elements, respectively. Moreover, for a matrix $M$, the notations $\|M\|_2$, $\|M\|_\infty$, $\|M\|_{1/1}$, $\|M\|_{\linf}$ refer to the induced 2-norm, induced $\infty$-norm, $\ell_1/\ell_1$ norm, and $\ell_\infty/\ell_\infty$ norm, respectively. For a $d\times d$ matrix $M$, We define $\|M\|_{\loff} = \|M\|_{\lone}-\sum_{i=1}^d|M_{ii}|$. For a vector $v$ and matrix $M$, the notations $\mathrm{supp}(v)$ and $\mathrm{supp}(M)$ are defined as the location of their nonzero elements. Given two sequences $f(n)$ and $g(n)$, the notation $f(n)\lesssim g(n)$ implies that there exists a constant $C<\infty$ that satisfies $f(n) \leq Cg(n)$. Moreover, $f(n)\asymp g(n)$ implies that $f(n)\lesssim g(n)$ and $g(n)\lesssim f(n)$. Given two scalars $a$ and $b$, the symbols $a\wedge b$ and $a\vee b$ are used to denote their minimum and maximum, respectively.

	\section{Problem Formulation}
	In this work, we consider a class of time-varying MRFs that can be expressed as families of exponential distributions, defined as:
	\begin{align}\label{exp}
		\bP(x_t;\theta_t^\star) = \exp\left\{\langle \theta^\star_t, \phi(x_t)\rangle-A(\theta^\star_t)\right\}\qquad \text{for}\quad t=0,\dots,T.
	\end{align}
	Here, $x_t\in \R^n$ is the random variable at time $t$, $\theta^\star_t$ is the (unknown) true \textit{canonical parameter} from the domain $\cD\subseteq\mathbb{R}^p$,
 the function $\phi : \mathbb{R}^n\to\mathbb{R}^p$ is the \textit{sufficient statistics}, and $A: \mathbb{R}^p\to\mathbb{R}$ is the \textit{log-partition} function, which normalizes the distribution. Special classes of time-varying MRFs include time-varying Gaussian MRFs (GMRFs) and time-varying Discrete MRFs, 
	respectively corresponding to multivariate Gaussian and discrete random processes.
	As will be explained later, due to the equivalence between MRFs and exponential families, Markov graphs can be systematically obtained from their canonical parameters $\{\theta^\star_t\}_{t=0}^T$~\citep{wainwright2008graphical}.
	
	An alternative parameterization of exponential families is via \textit{mean} or \textit{moment parameters}, i.e., the moments of the sufficient statistics $\mu_t(\theta^\star_t) = \mathbb{E}_{\theta^\star_t}[\phi(x_t)]\in\cR$, where $\cR$ is the so-called \textit{moment polytope}. Given the true canonical parameters $\{\theta^\star_t\}_{t=0}^T$, the mean parameters $\{\mu_t(\theta^\star_t)\}_{t=0}^T$ can be obtained via the \textit{forward mapping} $F:\cD\to\cR$, where $\mu_t(\theta_t) = F(\theta_t) = \nabla A(\theta_t)$. The conjugate (or Fenchel) duality can be used to define the \emph{backward mapping} $F^*:\cR\to\cD$ with $\theta_t(\mu_t) = F^*(\mu_t) = \nabla A^*(\mu_t)$, where $A^*$ is the conjugate dual of the log-partition function~\citep[Chapter 3]{wainwright2008graphical}. For well-known classes of MRFs, the sufficient statistics, canonical parameters, and backward and forward mappings can be characterized based on the underlying probability distribution.
	\begin{figure}
		\centering
		\includegraphics[width=10cm]{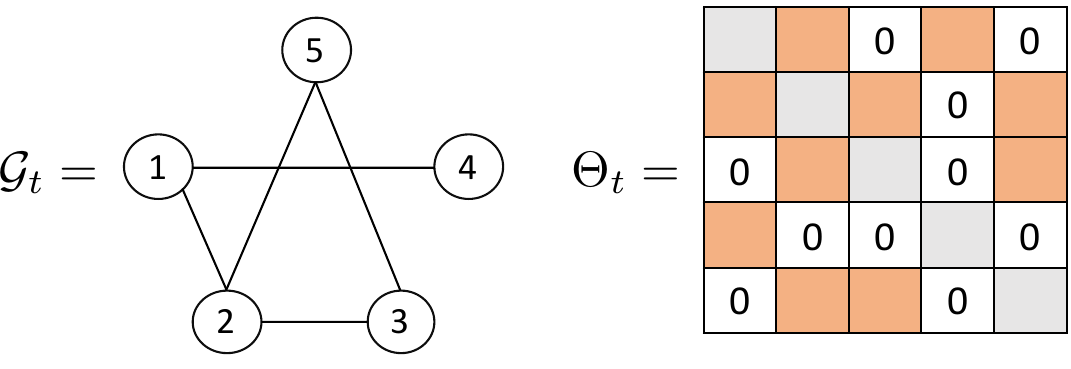}
		\caption{\label{fig_GMRF} (left) The Markov graph of a time-varying GMRF, (right) the sparsity pattern of the precision matrix.}
		\vspace{-3mm}
	\end{figure}
	\paragraph{Gaussian setting.} The canonical parameters in a Gaussian time-varying MRF (GMRF) correspond to the tuple of inverse covariance matrix $\Theta_t\in\bR^{n\times n}$ (also known as the precision matrix) and mean vector $\eta_t\in\bR^{n}$. In this setting, the exponential distribution~\eqref{exp} reduces to 
	\begin{align}\label{gmrf}
		\mathbb{P}(x_t) \!=\! \exp\left\{-\frac{1}{2}\langle\Theta^\star_t,x_tx_t^\top\rangle\!-\!A(\Theta^\star_t)\right\},
	\end{align}
	with sufficient statistics $x_tx_t^\top$, where without loss of generality we assumed that $\eta_t = \bE[x_t] = 0$. For GMRFs, the edge set of the Markov graph $\mathcal{G}_t$ coincides  with the off-diagonal nonzero elements of $\Theta^\star_t$~\citep{weiss2000correctness}. 
	Figure~\ref{fig_GMRF} shows the sparsity pattern of a precision matrix for a GMRF and its corresponding Markov graph. The moment parameter $\mu_t$ in this setting corresponds to the covariance matrix $\Sigma_t^\star = \bE[x_tx_t^\top]$. Accordingly, the forward and backward mappings $F:\cD\to\cR$ and $F^*:\cR\to\cD$ can be defined as $F(\Theta_t) = \Sigma_t = \Theta_t^{-1}$ and $F(\Sigma_t) = \Sigma_t^{-1}$ with $\cD = \{\Theta_t: \Theta_t = \Theta_t^\top, \Theta_t\succ 0\}$ and $\cR=\{\Sigma_t: \Sigma_t = \Sigma_t^\top, \Sigma_t\succ 0\}$.
	
	\paragraph{Discrete setting.} In discrete time-varying MRFs (DMRF), each element of the random variable $x_t$ takes a value from the discrete set $\mathcal{K}$ with cardinality $|\mathcal{K}| = K$. For simplicity, we assume that $\mathcal{K}\subseteq [0,1]$. In this case,~\eqref{exp} reduces to
	\begin{align}\label{dmrf}
		\mathbb{P}(x_t) \!=\! \exp\left\{\sum_{i\in[p], k\in\mathcal{K}}\theta^\star_{t;ik}\mathbb{I}[x_{t;i} = k]+\sum_{i,j\in[p], k,l\in\cK}\theta^\star_{t;ijkl}\mathbb{I}[x_{t;i} = k, x_{t;j} = l]-A(\theta^\star_t)\right\},
	\end{align}
	where $\theta^\star_{t;i\boldsymbol{\cdot}}\in\bR^{K}$ is a vector parameter for the node-wise indicator function $\mathbb{I}[x_{t;i} = k]$, and $\theta^\star_{t;ij\boldsymbol{\cdot}\boldsymbol{\cdot}}\in\bR^{K\times K}$ is a matrix parameter for the edge-wise indicator function $\mathbb{I}[x_{t;i} = k, x_{t;j} = l]$. The node- and edge-wise indicator functions form an (over-complete) sufficient statistics. In this case, the edge set of the Markov graph $\mathcal{G}_t$ corresponds to the nonzero matrix parameters $\theta^\star_{t;ij\boldsymbol{\cdot}\boldsymbol{\cdot}}$; in other words, $(i,j)\in E_t$ if and only if $\theta^\star_{t;ij\boldsymbol{\cdot}\boldsymbol{\cdot}} \not= 0_{K\times K}$.  The moment parameter, in this case, coincides with the node- and edge-wise marginal probabilities, defined as $\mu^\star_{t;ik} = \bE[\mathbb{I}[x_{t;i} = k]] = \bP(x_{t;i} = k)$ and $\mu^\star_{t;ijkl} = \bE[\mathbb{I}[x_{t;i} = k, x_{t;j} = l]] = \bP(x_{t;i} = k, x_{t;j} = l)$. Accordingly, the forward mapping $F:\cD\to\cR$ is a function that computes the node- and pair-wise marginal probabilities given the canonical parameters of the joint distribution. Moreover, the backward mapping $F^*:\cR\to\cD$ recovers the set of canonical parameters under which the joint probability distribution has node- and pair-wise marginal distributions that agree with the mean parameters. In this setting, $\cD = \{\{(\theta_{i\ccdot},\theta_{jv\ccdot\ccdot})\}_{i,j,v}|\theta_{i\ccdot}\in\bR^{K}, \theta_{jv\ccdot\ccdot}\in\bR^{K\times K}, \theta_{jv\ccdot\ccdot} = \theta_{jv\ccdot\ccdot}^\top, \theta_{jv\ccdot\ccdot} = \theta_{vj\ccdot\ccdot}, \forall i,j,v\in[p]\}$, and $\cR$ corresponds to the convex hull of all valid marginal distributions (also known as \textit{marginal polytope}). Unfortunately, obtaining the exact and compact representation of the marginal polytope is NP-hard~\citep{roughgarden2013marginals}, which leads to the intractability of computing the backward mapping for time-varying DMRFs~\citep{bresler2014hardness, montanari2015computational}. To circumvent this difficulty, a common approach is to resort to \textit{variational approximations} of the backward mapping via \textit{tree-reweighted entropy}~\citep{wainwright2003tree, wainwright2008graphical}, which will be discussed in detail in \S\ref{subsec_dmrf}.
	
	\subsection{Estimating the Canonical Parameters}
	In practice, the true mean parameters are rarely available and should be replaced by their \textit{empirical} analogs $\{\widehat{\mu}_t\}_{t=0}^T$, where $\widehat{\mu}_t = \frac{1}{N_t}\sum_{i=1}^{N_t}\phi\left(x^{(i)}_t\right)$ and $\left\{x^{(i)}_t\right\}_{i=1}^{N_t}$ are the samples at time $t$~\citep{jordan1999introduction}. 
	Given the empirical mean parameters and the backward mapping, one can estimate the canonical parameters from which the Markov graphs can be extracted. However, the backward mapping may be ill-defined or hard to derive explicitly, especially in the time-varying regime where the number of available samples at each time may be extremely small~\citep{wainwright2008graphical, zhou2010time}. For example, for the time-varying GMRF, the empirical moment parameter coincides with the sample covariance matrix $\widehat{\Sigma_t} = \frac{1}{N_t}\sum_{i=1}^{N_t}x^{(i)}_{t}{x^{(i)}_{t}}^\top$. Recalling the explicit form of the backward mapping, the precision matrix can be estimated as $\widehat{\Theta_t} = \widehat{\Sigma_t}^{-1}$. In time-varying settings, $N_t$ is often significantly smaller than $p$, which in turn implies that the sample covariance matrix is rank deficient and, therefore, not invertible. Under such circumstances, a consistent estimation of the canonical parameters is only possible by leveraging additional prior information on the structural or temporal properties of the MRF. For instance, many real-world problems have sparsely-~\citep{barabasi2004network, montanari2012graphical, ravikumar2023efficient} and/or smoothly-changing~\citep{sporns2004organization, yu2008small} structures. These structures can be promoted via different regularizers.

 For example, consider a sparsely-changing GMRF, where the time-varying precision matrices are sparse and change sparsely over time. The corresponding {$\ell_1$-regularized MLE}, also known as \textit{time-varying Graphical Lasso}~\citep{hallac2017network}, is defined as:
 \begin{subequations}\label{mle_reg_gmrf}
		\begin{align}
			\{\widehat{\Theta}_t\}_{t=0}^T =  \arg\!\sup_{\{{\Theta}_t\}_{t=0}^T}& \ \sum_{t=0}^{T}\left(\langle\Theta_t,\widehat{\Sigma}_t\rangle-\log\det(\Theta_t)\right)+\gamma_1 \sum_{t=0}^{T}\|\Theta_t\|_{\loff}+\gamma_2\sum_{t=0}^{T}\|\Theta_t- \Theta_{t-1}\|_{\lone}\\
			\text{s.t.}&\ \ \Theta_t\succ 0\qquad\quad t = 0,1,\dots, T
		\end{align}
	\end{subequations}
	where $\norm{\Theta_t}_1$ and $\norm{\Theta_t- \Theta_{t-1}}_{1/1}$ are convex relaxations of the ``$\ell_0$-norm'' that promote sparsity in the precision matrices and their differences. 
 However, $\ell_1$-regularized MLE suffers from inferior statistical properties, which is shown in our next example. 
	\begin{figure}
		\centering
		\subfloat[]{%
			\includegraphics[width=6.5cm]{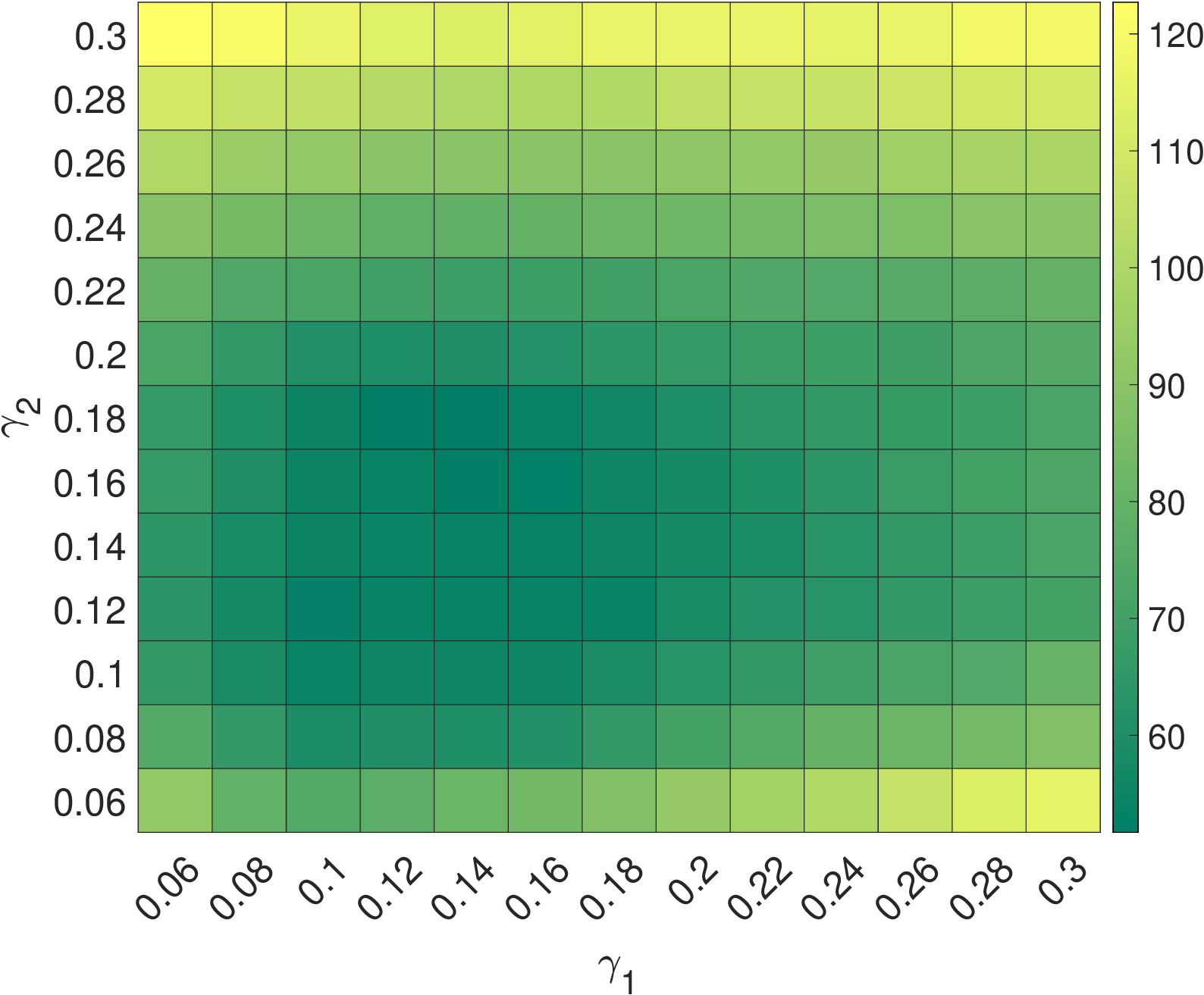}\label{fig_mle_heatmap}
		}
		\subfloat[]{
			\includegraphics[width=6.8cm]{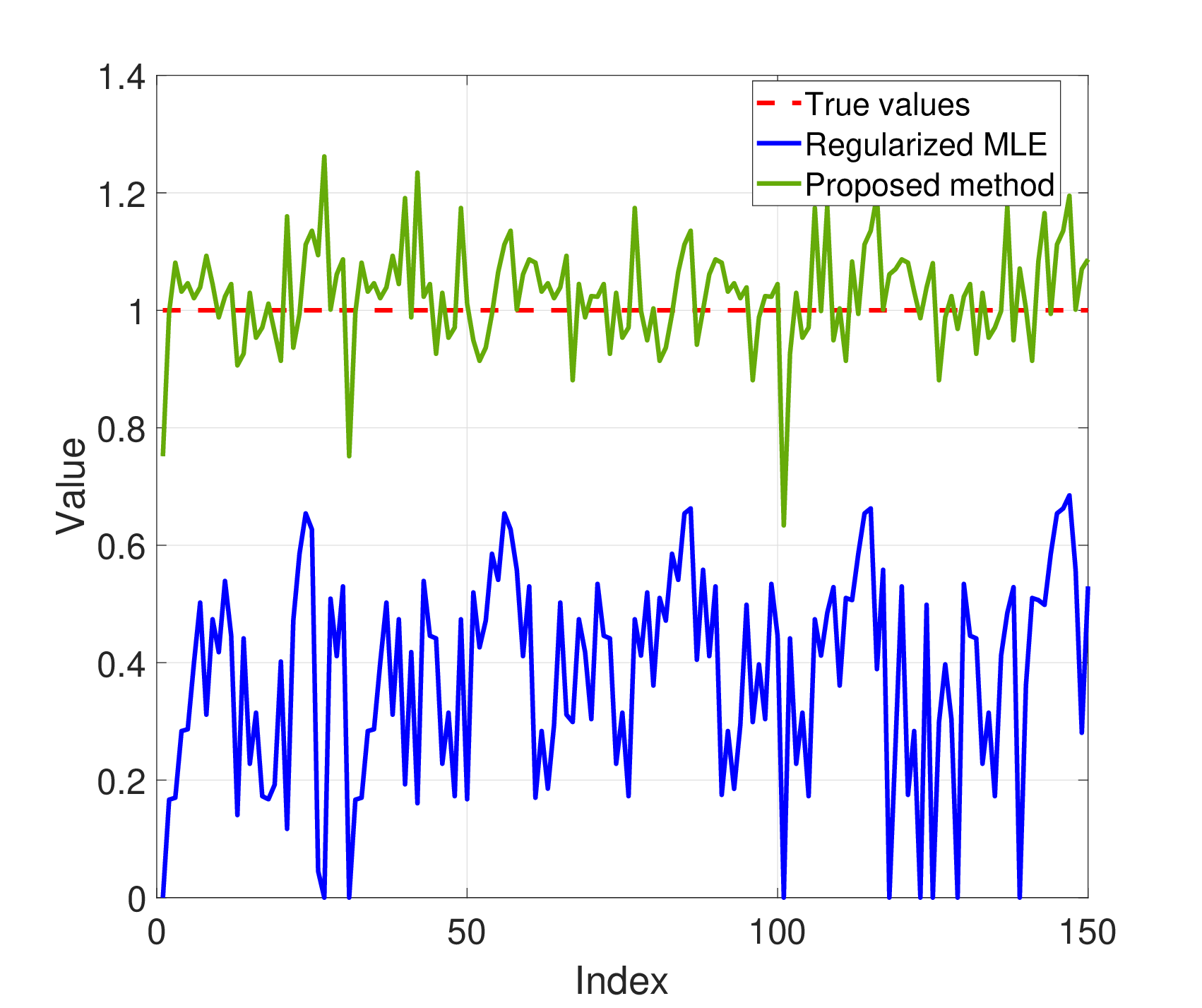}\label{fig_nonzeros_mle_small}
		}
		\caption{(a) The heatmap of the mismatch error. (b) The true and estimated nonzero elements of the precision matrix.}
	\end{figure}
	\paragraph{Example (regularized MLE for time-varying GMRF).} Consider an instance of time-varying GMRF with $T = 5$ and $n = 25$, where $\{\Theta^\star_t\}_{t=0}^4 \in\mathbb{R}^{25\times 25}$ are randomly generated symmetric and sparse matrices. At each time $t = 0,\dots,4$, the true precision matrix $\Theta^\star_t$ has exactly 30 off-diagonal elements with value one (while maintaining its symmetry), and the remaining off-diagonal entries are set to zero. Moreover, the diagonal entries $\Theta^\star_{t;ii}$ are set to $1+\sum_{j\not=i}\Theta^\star_{t;ij}$ to ensure $\Theta^\star_t\succ 0$ for every $t = 0,\dots, 4$. At every time, 5 nonzero off-diagonal elements are changed to zero, and 5 zero elements are set to one. The sample covariance $\widehat{\Sigma}_t$ for every $t = 0,\dots, 4$ is obtained by collecting 500 samples from the Gaussian distribution with the constructed precision matrices. 
 Figure~\ref{fig_mle_heatmap} illustrates a heatmap of the \textit{mismatch error}, i.e., the total number of false positives and false negatives in the sparsity patterns of the true and estimated precision matrices and their differences, for different values of the regularization coefficients. It can be seen that after an exhaustive search over the regularization coefficient space, the best achievable mismatch error is in the order of 50. Note that simply predicting a fully sparse precision matrix has a mismatch error of 190. In light of this, the estimated parameters reveal little information about the true structure of the time-varying GMRF. Moreover, Figure~\ref{fig_nonzeros_mle_small} shows the concatenation of the nonzero elements in the true precision matrices (dashed red line), and their corresponding estimated values via the regularized MLE (blue) and our proposed method (green). It can be seen that, even when the sparsity pattern of the elements is correctly recovered, the estimated nonzero entries via regularized MLE are ``shrunk'' toward zero due to the shrinking effect of $\ell_1$ regularizer, thereby incurring a {substantial} bias. In contrast, our proposed method circumvents this undesirable bias by directly employing the $\ell_0$ regularizer instead.

	\section{Related Work}

\label{sec:literature}

\paragraph{\bf Time-varying MRF.} 
The problem of inferring time-varying MRFs can be traced back to Kalman filters, where the goal is to predict a random signal evolving over time by filtering out the observational noise~\citep{kalman1960new, brown1992introduction}. More recent results have studied the non-asymptotic inference of time-varying MRF with side information, such as sparsity and smoothness. 
In addition to the time-varying Grahical Lasso introduced in \eqref{mle_reg_gmrf}, 
a recent line of works has studied the inference of smoothly-changing MRFs~\citep{kolar2011time,greenewald2017time,zhou2010time}. These methods rely on kernel methods, where the empirical mean parameter at any given time is estimated as a weighted average of the samples over time, and the weights are collected from a predefined kernel. However, these methods do not leverage the prior information about the sparsity of the parameter differences. With the goal of addressing this deficiency, several works have studied the inference of sparsely-changing MRF (also known as sparse \textit{differential networks})~\citep{wang2018fast, zhao2014direct, liu2017learning}. However, the main drawback of these methods is that they only estimate the parameter differences, and their theoretical guarantees are restricted to problems with two time steps ($T=1$). Similarly, regression-based approaches have been proposed for change point detection problems \citep{kolar2012estimating,roy2017change} with MRFs and two time periods, assuming the sparsity pattern of all entries of the precision matrices change at the same time. In contrast, \cite{wang2014inference} have studied the inference of sparse MRFs given an index variable under the assumption that the sparsity pattern is invariant, whereas \cite{geng2019partially} have assumed that the precision matrix is a linear function of the index variable.

\paragraph{Sparsity-promoting optimization.} 
Unfortunately, solving statistical inference problems with sparsity is often an NP-hard discrete optimization problem \citep{chickering1996learning,natarajan1995sparse}, and exact methods requiring exhaustive search \citep{lauritzen1996graphical} are impractical.
Perhaps the most popular approaches are \textit{Fused Graphical Lasso} and \textit{Group Graphical Lasso}~\citep{danaher2014joint}, which propose to relax $\|\theta_t\|_0$ and $\|\theta_t-\theta_{t-1}\|_0$ with their $\ell_1$-approximations. 
The $\ell_1$-norm proxy for sparsity is by now a standard approach in inference with graphical models \citep{banerjee2008model,friedman2008sparse,meinshausen2006high}. However, such relaxation results in inferior statistical performance than its exact, non-convex counterpart \citep{bertsimas2019certifiably}. Moreover, MLE-based estimators with $\ell_1$ regularization are statistically consistent only under stringent conditions, while in some cases, their nonconvex counterparts do not require such strong restrictions \citep{lam2009sparsistency}.

Also closely related to our setting is the class of optimization problems of the form 
\small\begin{align}\label{pairwise}
\min_{ \theta\in [\ell,u]^p}\;& \|\theta\|_0+\sum_{i=1}^p\sum_{j=i+1}^p g_{ij}(\theta_i-\theta_j),
\end{align}\normalsize
for given one-dimensional functions $g_{ij}:\R\to \R$,.  If functions $g_{ij}$ are convex, then problem \eqref{pairwise} admits pseudo-polynomial time algorithm \citep{ahuja2004cut,bach2019submodular}. Moreover, convex relaxations that deliver near-optimal solutions for \eqref{pairwise} were proposed for the special case of convex quadratic $g$ functions \citep{atamturk2018sparse}. If, additionally, we have $\ell=0$ and $u=\infty$, then problem \eqref{pairwise} is, in fact, solvable in strongly polynomial time \citep{atamturk2018strong}. As another special case, if $\ell=-\infty$, $u=\infty$, and $g_{ij}(x)=0$ whenever $j>i+1$, then the problem \eqref{pairwise} with convex quadratic $g$ functions can be solved in quadratic time~\citep{liu2022graph}. On the other hand, problem \eqref{pairwise} is much more challenging for non-convex $g$: if $g(x)=\mathbbm{1}\left\{x\neq 0\right\}$, as is the case for sparsely-changing MRFs, then problem \eqref{pairwise} is NP-hard even if the term $\|\theta\|_0$ is dropped from the objective \citep{hochbaum2001efficient}. Nonetheless, as we show in this paper, problem \eqref{pairwise} can be solved efficiently in the context of time-varying MRFs, where $g_{ij}(x)=0$ whenever $j>i+1$. 

	\vspace{2mm}
	\section{Proposed Method}
 We propose the following class of constrained optimization problems:

	\vspace{6mm}
	\noindent\fbox{\parbox{\textwidth-3mm}{
	\begin{equation}\label{generalopt1}\tag{ProxGL}
		\begin{aligned}
			\{\widehat\theta_t\}_{t=0}^T = \arg\min_{\{{\theta}_t\}_{t=0}^T}&\gamma \!\!\!\!\underbrace{\sum_{t=0}^T\norm{\theta_t}_0}_{\textbf{absolute regularizer}}+(1-\gamma)\underbrace{\sum_{t=1}^T\norm{\theta_t-\theta_{t-1}}^q_q}_{\textbf{temporal regularizer}},\\
			\mathrm{s.t.}&\underbrace{\norm{\theta_t-\widetilde F^*(\widehat\mu_t)}_{\infty}}_{\textbf{backward mapping deviation}}\leq \lambda_t \qquad \forall t = 0,\dots, T,
		\end{aligned}
	\end{equation}
}}

\vspace{6mm}
	\noindent where $q\geq 0$ in the temporal $\ell_q$-regularizer is chosen to match the temporal changes in the canonical parameters. For example, $q=0$ is a natural choice for sparsely-changing MRF where only a few elements of the canonical parameters change in consecutive times, while $q=2$ is suitable for smoothly-changing MRFs where canonical parameters change smoothly over time. The parameter $0\leq\gamma\leq1$ is a regularization parameter that is used to balance the sparsity level and the temporal changes. Indeed, a larger value for $\gamma$ would lead to sparser parameters, whereas a smaller $\gamma$ would result in smaller temporal changes in the parameters. In the above optimization, the \textit{backward mapping deviation} controls the deviation of the canonical parameters from an approximate backward mapping $\widetilde F^*(\widehat\mu_t)$, which we assume can be obtained directly from the data. In \S\ref{sec_stats}, we provide a detailed discussion on how to obtain such approximate backward mapping for different classes of time-varying MRFs.
 Despite its nonconvexity, our next theorem shows that \ref{generalopt1} can be solved efficiently. Recall that $T$ is the number of time steps, and $p$ is the dimension of the canonical parameter at each time. 
	
	\begin{sloppypar}
 \begin{theorem}[Efficient algorithm]\label{thm_runtime}
 For all values of the regularization parameter $\gamma\in [0,1]$ and any $q\in \{0\}\cup [1,\infty)$, the entire solution path of~\ref{generalopt1} can be obtained in at most $\mathcal{O}(pT^3)$ time on a single thread.
\end{theorem}
	\end{sloppypar}
	
	The proof of Theorem~\ref{thm_runtime} is provided in \S\ref{sec_algorithm} together with the detailed implementation of our proposed algorithm.
	To the best of our knowledge, Theorem~\ref{thm_runtime} is the first result showing that the entire solution path of time-varying MRFs can be estimated in strongly polynomial time. 
	The availability of the solution path facilitates the automated fine-tuning of the sparsity level and/or smoothness of the obtained canonical parameters via cross-validation or by solving a bilevel optimization, where the estimated canonical parameters on the training set are parameterized as functions of $\gamma$ and optimized on the validation set~\citep{sinha2014bilevel, mackay2019self}.
	Our algorithm for solving~\ref{generalopt1} relies on two key properties of our proposed formulation: First,~\ref{generalopt1}  is decomposable into $p$ independent subproblems over different coordinates of $\{\theta_t\}_{t=0}^T$. This guarantees that the complexity of our algorithm depends linearly on the dimension $p$. Moreover, we propose a dynamic programming approach that can solve each subproblem in $O(T^3)$ for all values of $\gamma$.

	Given the efficient solvability of~\ref{generalopt1}, our next goal is to show that the obtained solution indeed enjoys a small estimation error. To this goal, we formally define two classes of temporal structures for the canonical parameters, namely sparsely-changing MRFs and smoothly-changing MRFs. 
	\begin{definition}[Sparsely-changing MRF]\label{assump:sparse}
		A time-varying MRF is sparsely-changing with parameter $D_0\geq 0$ if $\norm{\theta^\star_{t}-\theta^\star_{t-1}}_0\leq D_0$ for every $t = 1,\dots,T$.
	\end{definition}

Intuitively, we say that ``a time-varying MRF is sparsely-changing'' if it is sparsely-changing with a parameter that satisfies $D_0\ll p$. We note that, if $\norm{\theta^\star_t}_0\leq S_0$ for every $t$, then $D_0\leq 2S_0$. In other words, time-varying MRFs with sparse parameters are also sparsely-changing with parameter $2S_0$. 
	
	\begin{sloppypar}
		\begin{definition}[Smoothly-changing MRF]\label{assump:smooth}
			A time-varying MRF is smoothly-changing with parameters $(q, D)$ if $\underset{1\leq i\leq n}{\max}\left\{\sum_{t=1}^T|\theta^\star_{t;i}-\theta^\star_{t-1;i}|^q\right\}\leq D$.
		\end{definition}
	\end{sloppypar}
	Similarly, We say ``a time-varying MRF is smoothly-changing'' if it is smoothly-changing with $q\geq 1$ and $D\ll T$. Indeed, a time-varying MRF can be simultaneously sparsely-changing and smoothly-changing. 
In such settings, only a few elements of the canonical parameters change over time and the changes are smooth. 
 However, to streamline the presentation, we study these two scenarios separately. 

Let $\mathcal{S}_t$ denote the set of indices corresponding to the nonzero elements of $\theta_t^\star$ for every $t=0,\dots,T$. Similarly, let $\mathcal{D}_t$ denote the set of indices corresponding to the nonzero elements of $\theta_t^\star-\theta_{t-1}^\star$ for every $t=1,\dots,T$. Our next theorem provides conditions under which an optimal solution of~\ref{generalopt1} enjoys a small estimation error and correct sparsity pattern.
	\begin{sloppypar}
		\begin{theorem}[Estimation error and sparsistency]\label{thm_constrained}
			 Suppose that the parameter $\lambda_t$ in \ref{generalopt1} satisfies
			\begin{align}\tag{A1}
				\norm{\theta_t^\star-\widetilde{F}^*(\widehat{\mu}_t)}_\infty\leq \lambda_t, \quad \forall t = 0,\dots, T.\label{asp_A1}
			\end{align}
			Then the following statements hold:
			\begin{itemize}
				\item \textbf{Estimation error.} We have
				\begin{align}
				\norm{\widehat{\theta}_t-\thetatrue}_\infty\leq 2\lambda_t, \quad \forall t = 0,\dots,T.\label{eq_error}
				\end{align}
				\item \textbf{Sparsistency for smoothly-changing MRF.} Suppose that the time-varying MRF is {smoothly-changing} with parameters $(q,D)$ for some $q\geq 1, D\geq 0$. Additionally, suppose that $0<\gamma<1/(1+D)$ and $2\lambda_t\leq\min_{i\in\mathcal{S}_t}|\theta_{t;i}^\star|$. Then, the optimal solution of~\ref{generalopt1} with temporal $\ell_q$-regularizer satisfies:
				\begin{align}\label{eq_sparse_q}
					\mathrm{supp}(\widehat{\theta}_t) = \mathrm{supp}({\theta}^\star_t), \qquad \forall t = 0,\dots,T.
				\end{align}
				\item \textbf{Sparsistency for sparsely-changing MRF.} Suppose that the time-varying MRF is {sparsely-changing} with parameter $D_0\geq 0$. Moreover, suppose that $2\lambda_t\leq\min_{i\in\mathcal{S}_t}|\theta_{t;i}^\star|$ and  $2\lambda_t+2\lambda_{t-1}\leq \min_{i\in\mathcal{D}_t}|\theta_{t;i}^\star-\theta_{t-1;i}^\star|$. Then, with any choice of $0<\gamma<1$, the optimal solution of~\ref{generalopt1} with temporal $\ell_0$-regularizer satisfies:
				\begin{equation}\label{eq_sparse_0}
					\begin{aligned}
					& \mathrm{supp}(\widehat{\theta}_t) = \mathrm{supp}({\theta}^\star_t), && \qquad \forall t = 0,\dots,T,\\
					& \mathrm{supp}(\widehat{\theta}_t-\widehat{\theta}_{t-1}) = \mathrm{supp}({\theta}^\star_t-{\theta}^\star_{t-1}), && \qquad \forall t = 1,\dots,T.
					\end{aligned}
				\end{equation}
			\end{itemize}
		\end{theorem}
	\end{sloppypar}

	The above theorem shows that if the true canonical parameters remain feasible for~\ref{generalopt1} with sufficiently small $\lambda_t$, then the estimated parameters have small errors and correct sparsity patterns. Moreover, if the time-varying MRF is sparsely-changing, then the solution of~\ref{generalopt1} with a temporal $\ell_0$-regularizer recovers the correct sparsity pattern of the individual parameters, as well as their temporal changes. In \S\ref{sec_stats}, we use the above theorem to characterize the sample complexity of our proposed framework for two classes of time-varying MRFs, namely GMRFs and DMRFs.
Next, we present the proof of Theorem~\ref{thm_constrained}.
\vspace{2mm}

\noindent {\bf Proof of Theorem~\ref{thm_constrained}.}
We first present the proof of estimation error bound~\eqref{eq_error}. 
	\paragraph{\it \underline{Estimation error bound.}}
	The solution of~\ref{generalopt1} satisfies
	\begin{align*}
		\norm{\widehat{\theta}_t-\theta^\star_t}_\infty\leq \norm{\thetaest-\bmest}_\infty+\norm{\thetatrue-\bmest}_\infty\leq 2\lambda_t,
	\end{align*}
	where in the first inequality, we used the triangle inequality, and in the second inequality, we used the assumption that the true canonical parameter $\{\theta^\star_t\}_{t=0}^T$ is feasible for~\ref{generalopt1}.
	
	\paragraph{\it \underline{Sparsistency.}} Next, we show the sparsistency of the estimated parameters for any $q\in\{0\}\cup [1,\infty)$. First, we prove $\supp(\thetaest) = \supp(\thetatrue)$ for any $q\geq 1$. To this goal, we show $\supp(\thetaest)\subseteq \supp(\thetatrue)$ and $\supp(\thetatrue)\subseteq\supp(\thetaest)$. For any $i\in\mathcal{S}_t$, we have
	\begin{align}
		\left|\widehat \theta_{t;i}\right| &= \left|\widehat \theta_{t;i}-\theta_{t;i}^\star+\theta_{t;i}^\star\right|\nonumber\\
		&\geq \left|\theta_{t;i}^\star\right| - \left|\widehat \theta_{t;i}-\theta_{t;i}^\star\right|\nonumber\\
		&\geq \left|\theta_{t;i}^\star\right|-2\lambda_t\nonumber\\
		&>0\nonumber,
	\end{align}
	where in the second inequality we used the upper bound on the estimation error, and in the last inequality we used the assumption $ 2\lambda_t\leq \min_{i\in\mathcal{S}_t}|\theta^\star_{t;i}|$. This implies that $\supp(\thetatrue)\subseteq\supp(\thetaest)$. To prove $\supp(\thetaest)\subseteq\supp(\thetatrue)$, we first rewrite~\ref{generalopt1} as follows:
	\begin{equation}\label{generaloptrepeated}
\begin{aligned}
		 \{\widehat{\theta}_t\}_{t=0}^T = \arg\min_{\{{\theta}_t\}_{t=0}^T}&\gamma \sum_{t=0}^T\sum_{i=1}^n\mathbb{I}[\theta_{t;i} \neq 0]+(1-\gamma)\sum_{t=1}^T\sum_{i=1}^n\left|\theta_{t;i}-\theta_{t-1;i}\right|^q\\
			\mathrm{s.t.}\; &\left|\theta_{t;i}- \left[\widetilde F^*(\widehat\mu_t)\right]_i\right|\leq \lambda_t\qquad \forall t=0,\dots,T,\; i=1,\dots,p.
		\end{aligned}
	\end{equation}
We make the following observation. First, due to the use of the infinity norm for the backward mapping deviation, problem \eqref{generaloptrepeated} decomposes into $p$ independent subproblems, one for each coordinate of $\theta_t$.
	On the other hand, due to the optimality of $\{\thetaest\}_{t=0}^T$ and the feasibility of $\{\thetatrue\}_{t=0}^T$, we have for every $1\leq i\leq p$
	\begin{align}
		&(1\!-\!\gamma)\sum_{t=0}^T\mathbb{I}\{\widehat\theta_{t;i}\not=0\}\!+\!\gamma\sum_{t=1}^T{\left|\widehat{\theta}_{t;i}\!-\!\widehat{\theta}_{t-1;i}\right|^q}\!\leq\! (1-\gamma)\sum_{t=0}^T\mathbb{I}\{\theta^\star_{t;i}\not=0\}\!+\!\gamma\sum_{t=1}^T{\left|{\theta}^\star_{t;i}-{\theta}^\star_{t-1;i}\right|^q}\nonumber\\
		\implies&(1-\gamma)\sum_{t=0}^T\mathbb{I}\{\widehat\theta_{t;i}\not=0\}
		\leq (1-\gamma)\sum_{t=0}^T\mathbb{I}\{\theta^\star_{t;i}\not=0\}+\gamma D\nonumber\\
		\implies & (1\!-\!\gamma)\!\sum_{t=0}^T\mathbb{I}\{\widehat\theta_{t;i}\!\not=\!0\}\mathbb{I}\{i\!\in\!\mathcal{S}_t\} \!+\!(1\!-\!\gamma)\!\sum_{t=0}^T\mathbb{I}\{\widehat\theta_{t;i}\!\not=\!0\}\mathbb{I}\{i\!\in\![n]\backslash\mathcal{S}_t\}
		\!\leq\! (1\!-\!\gamma)\!\sum_{t=0}^T\!\left|{\theta}^\star_{t;i}\right|_0\!+\!\gamma D_q\nonumber\\
		\implies & (1-\gamma)\sum_{t=0}^T\mathbb{I}\{\widehat\theta_{t;i}\not=0\}\mathbb{I}\{i\in[n]\backslash\mathcal{S}_t\}
		\leq \gamma D\nonumber\\
		\implies & \sum_{t=0}^T\mathbb{I}\{\widehat\theta_{t;i}\not=0\}\mathbb{I}\{i\in[n]\backslash\mathcal{S}_t\}
		<1\nonumber
	\end{align}
	\begin{sloppypar}
		\noindent where the second inequality follows from the assumption $\sum_{t=1}^T\sum_{i\in[p]}|\theta_{t;i}^*-\theta_{t-1;i}^*|^q\leq D$. Moreover, the fourth inequality follows from $\supp(\thetatrue)\subseteq\supp(\thetaest)$. The last inequality follows from the assumption $0<\gamma<1/(1+D)$. This implies that $ \mathbb{I}\{\widehat\theta_{t;i}\not=0\}\mathbb{I}\{i\in[n]\backslash\mathcal{S}_t\} = 0$ for every $t$ and $i$, and hence, we have $\supp(\thetaest)\subseteq \supp(\thetatrue)$. 
	\end{sloppypar}
	
	Finally, we show that for sparsely-changing MRFs,~\ref{generalopt1} with temporal $\ell_0$-regularizer satisfies $\mathrm{supp}(\widehat{\theta}_t) = \mathrm{supp}({\theta}^\star_t)$ and $\mathrm{supp}(\widehat{\theta}_t-\widehat{\theta}_{t-1}) = \mathrm{supp}({\theta}^\star_t-{\theta}^\star_{t-1})$. An argument identical to $q\geq 1$ can be invoked to show that $\supp(\thetatrue)\subseteq\supp(\thetaest)$. Similarly, given any $i\in\mathcal{D}_t$, one can write
	\begin{align}
		\left|\widehat \theta_{t;i}-\widehat \theta_{t-1;i}\right| &= \left|\widehat \theta_{t;i}-\theta_{t;i}^*+\theta_{t;i}^*-\theta_{t-1;i}^*+\theta_{t-1;i}^*-\widehat \theta_{t-1;i}\right|\nonumber\\
		&\geq \left|\theta_{t;i}^*-\theta_{t-1;i}^*\right|-\left|\widehat \theta_{t;i}-\theta_{t;i}^*\right|-\left|\widehat \theta_{t-1;i}-\theta_{t-1;i}^*\right|\nonumber\\
		&\geq \left|\theta_{t;i}^*-\theta_{t-1;i}^*\right|-2\lambda_t-2\lambda_{t-1}\nonumber\\
		&>0,\nonumber
	\end{align} 
	where the last inequality is due to our assumption $2\lambda_t+2\lambda_{t-1}\leq \min_{i\in\mathcal{D}_t}|\theta_{t;i}^\star-\theta_{t-1;i}^\star|$. This implies that $\mathrm{supp}({\theta}^\star_t-{\theta}^\star_{t-1})\subseteq \mathrm{supp}(\widehat{\theta}_t-\widehat{\theta}_{t-1})$. On the other hand, due to the optimality of $\{\thetaest\}_{t=0}^T$ and the feasibility of $\{\thetatrue\}_{t=0}^T$, we have 
	\begin{align}
		&(1-\gamma)\sum_{t=0}^T\|\widehat\theta_t\|_0+\gamma\sum_{t=1}^T\|\widehat\theta_t-\widehat\theta_{t-1}\|_0\leq (1-\gamma)\sum_{t=0}^T\|\theta_t^*\|_0+\gamma\sum_{t=1}^T\|\theta_t^*-\theta_{t-1}^*\|_0\nonumber\\
		\implies&(1\!-\!\gamma)\sum_{t=0}^T\!\left(\sum_{i\in[n]\backslash\mathcal{S}_t}\mathbb{I}\{\widehat\theta_{t;i}\not=0\}\!+\!\sum_{i\in\mathcal{S}_t}\mathbb{I}\{\widehat\theta_{t;i}\not=0\}\right)\nonumber\\
		&\!+\!\gamma\sum_{t=1}^T\left(\sum_{i\in[n]\backslash\mathcal{D}_t}\!\mathbb{I}\{\widehat\theta_{t;i}-\widehat\theta_{t-1;i}\not=0\}\!+\!\sum_{i\in\mathcal{D}_t}\!\mathbb{I}\{\widehat\theta_{t;i}-\widehat\theta_{t-1;i}\not=0\}\right)\nonumber\\
		& \leq (1-\gamma)\sum_{t=0}^T\sum_{i\in\mathcal{S}_t}\mathbb{I}\{\theta^\star_{t;i}\not=0\}+\gamma\sum_{t=1}^T\sum_{i\in\mathcal{D}_t}\mathbb{I}\{\theta^\star_{t;i}-\theta^\star_{t-1;i}\not=0\}\nonumber\\
		\implies & (1-\gamma)\sum_{t=0}^T\sum_{i\in[n]\backslash\mathcal{S}_t}\mathbb{I}\{\widehat\theta_{t;i}\not=0\}+\gamma\sum_{t=1}^T\sum_{i\in[n]\backslash\mathcal{D}_t}\mathbb{I}\{\widehat\theta_{t;i}-\widehat\theta_{t-1;i}\not=0\}\leq 0\nonumber
	\end{align}
	where the last inequality follows from $\mathrm{supp}(\theta^*_t)\subseteq\mathrm{supp}(\widehat{\theta}_t)$ and $\mathrm{supp}(\theta_{t}^*-\theta_{t-1}^*)\subseteq \mathrm{supp}(\widehat\theta_{t}-\widehat\theta_{t-1})$. Due to $0<\gamma<1$, the above inequality implies that $\widehat\theta_{t;i} = 0$ for every $t = 0,\dots, T$ and $i\in[n]\backslash\mathcal{S}_t$, and $\widehat\theta_{t;i}-\widehat\theta_{t-1;i} = 0$ for every $t=1,\dots, T$ and $i\in[n]\backslash\mathcal{D}_t$. Therefore, we have $\mathrm{supp}(\widehat{\theta}_t)\subseteq \mathrm{supp}(\theta^*_t)$ and $\mathrm{supp}(\widehat\theta_{t}-\widehat\theta_{t-1})\subseteq\mathrm{supp}(\theta_{t}^*-\theta_{t-1}^*)$. This completes the proof.$\hfill\square$

\section{Efficient Algorithm}\label{sec_algorithm}
In this section, we describe the proposed algorithm for solving~\ref{generalopt1}, assuming that the approximate backward mapping $\widetilde F^*(\widehat\mu_t)$ is known. In the next section, we study efficient ways for obtaining $\widetilde F^*(\widehat\mu_t)$. Recall that, due to the use of the infinity norm for the backward mapping deviation, problem \eqref{generaloptrepeated} decomposes into $p$ independent subproblems, one for each coordinate of $\theta_t$. Moreover, if $\gamma=1$, then problem \eqref{generaloptrepeated} can be solved trivially, by setting $\theta_{t;i}=0$ whenever that solution is feasible and setting $\theta_{t;i}=\left[F^*(\widehat\mu_t)\right]_i$ otherwise. Thus, we assume without loss of generality that $p=1$ (so we omit subscripts $i$) and $\gamma<1$, and focus on solving the problem 
\begin{equation}\label{eq:opt}
	\begin{aligned}
	 \min_{\{{\theta}_t\}_{t=0}^T}&\bar \gamma \sum_{t=0}^T\mathbb{I}[\theta_t \neq 0]+\sum_{t=1}^T\left|\theta_t-\theta_{t-1}\right|^q\\
		\mathrm{s.t.}\;&\ell_t\leq\theta_t\leq u_t\qquad \forall t=0,\dots,T,
	\end{aligned}
\end{equation}
where $\bar \gamma=\gamma/(1-\gamma)$, $\ell_t=-\lambda_t+\widetilde F^*(\widehat\mu_t)$ and $u_t=\lambda_t+\widetilde F^*(\widehat\mu_t)$.	First, in \S\ref{sec:fixedparam}, we discuss how to solve \eqref{eq:opt} for a given value of the parameter $\bar\gamma$. We note that \cite{fattahi2021scalable} proposed a method for solving~\eqref{eq:opt} with $q=0$ by casting it as a shortest path problem, which is a special case of our proposed dynamic programming approach. Then, in \S\ref{sec:parametric}, we give an algorithm to solve problem \eqref{eq:opt} for all values of $\bar \gamma$. Finally, in \S\ref{sec:subproblem} we discuss further refinements of the algorithm. Overall, the main result of this section is an algorithm that solves \eqref{eq:opt} in strongly polynomial time, summarized in Proposition~\ref{prop:complexity} below.
\begin{proposition}\label{prop:complexity}
	The entire solution path of \eqref{eq:opt} for all values of $0\leq \bar{\gamma}\leq \infty$ and any fixed $q\in \{0\}\cup [1,\infty)$ can be computed in $\mathcal{O}(T^3)$.
\end{proposition}

\subsection{Fixed parameter solution}\label{sec:fixedparam}

A key insight to solve \eqref{eq:opt} is that if $\theta_t=0$ for any index $t$, then \eqref{eq:opt} decomposes into two independent problems. This key observation gives rise to a dynamic programming approach for solving~\eqref{eq:opt}. Given $0\leq a< b\leq T+1$, define function $v(a,b)$ to be the best objective value of \eqref{eq:opt} from time period $a$ to $b-1$ assuming that $\theta_{a-1}=\theta_b=0$ (if such variables exist), that is,
\begin{equation}\label{eq_v}
\begin{aligned}
    v(a,b)=\min_{\{{\theta}_t\}_{t=a}^{b-1}}&\bar \gamma \sum_{t=a}^{b-1}\mathbb{I}[\theta_t \neq 0]\!+\!\!\sum_{t=a+1}^{b-1}\left|\theta_t\!-\!\theta_{t-1}\right|^q\!+\!\mathbb{I}[a>0](\theta_a-0)^q\!+\!\mathbb{I}[b<T+1]|\theta_{b-1}-0|^q\\
		\mathrm{s.t.}\;&\ell_t\leq\theta_t\leq u_t\qquad \forall t=a,\dots,b-1.
\end{aligned}
\end{equation} 
Clearly, the optimal objective value of problem \eqref{eq:opt} is $v(0,T+1)$.
Moreover, let $f(a,b)$ denote the best objective value of \eqref{eq:opt} from indexes $a$ to $b-1$ while ignoring the absolute regularizer terms, that is,
\begin{equation}\label{eq:defCont}
\begin{aligned}
	f(a,b)=\min_{\{{\theta}_t\}_{t=a}^{b-1}}&\sum_{t=a+1}^{b-1}\left|\theta_t-\theta_{t-1}\right|^q+\mathbb{I}[a>0]\left|\theta_a-0\right|^q+\mathbb{I}[b<T+1]\left|\theta_{b-1}-0\right|^q\\
	\mathrm{s.t.}\;&\ell_t\leq\theta_t\leq u_t\qquad \forall t=a,\dots,b-1.
\end{aligned} 
\end{equation}
By convention, we also let $f(a,a)=v(a,a)=0$ for all $a=0,\dots,T+1$. 
In \S\ref{sec:subproblem}, we discuss the complexity of solving \eqref{eq:defCont}.

Functions $v$ and $f$ are related through the following  recursion:
\begin{lemma}
    For any $0\leq a<b\leq T+1$, we have
    \begin{equation}\label{eq:recursion}
	v(a,b)=\min\Big\{\underbrace{ \min_{t\in\{a,\dots,b-1\}:\; 0\in [\ell_t,u_t]}v(a,t)+v(t+1,b)}_{\text{Cost of setting }\theta_t=0},\; \underbrace{f(a,b)+\bar \gamma(b-a)}_{\text{Cost of setting }\theta_t\neq 0,\;  a\leq t\leq b-1}\Big\}.
\end{equation}
\end{lemma}
\begin{proof}
The optimal solution corresponding to $v(a,b)$ either satisfies $\theta_t\not=0$ for all $a\leq t\leq b-1$, or $\theta_t=0$ for some $a\leq t\leq b-1$. If $\theta_t\not=0$ for all $a\leq t\leq b-1$, then we have $v(a,b) = f(a,b)+\bar \gamma \sum_{t=a}^{b-1}\mathbb{I}[\theta_t \neq 0] = f(a,b)+\bar \gamma(b-a)$. Otherwise, the problem~\eqref{eq_v} decomposes at the break-point $t$ with $\theta_t = 0$ which results in $v(a,b) = v(a,t)+v(t+1,b)$. Therefore, $v(a,b)$ can be obtained by taking the minimum of $f(a,b)+\bar \gamma(b-a)$ and $v(a,t)+v(t+1,b)$ for the best choice of the break-point $t$.
\end{proof}

In our subsequent arguments, $EO$ stands for ``evaluation oracle" and denotes the complexity of solving \eqref{eq:defCont}. Moreover, $EO_2$ stands for the complexity of solving \eqref{eq:defCont} for all combinations of parameters $0\leq a<b\leq T+1$. Clearly, $EO_2=\mathcal{O}(T^2 EO)$, by simply calling the evaluation oracle $T^2$ times, although more efficient methods may exist.
\begin{lemma}
    For a given value of $\bar{\gamma}$, \eqref{eq:opt} can be solved with complexity $\mathcal{O}\big(EO_2+T^3\big)$.
\end{lemma}
\begin{proof}
From the recursion \eqref{eq:recursion}, it follows that if all values of $v(a,t)$ and $v(t+1,b)$ for all $t\in \{a,\dots,b-1\}$ are available, and $f(a,b)$ is available as well, then computing $v(a,b)$ reduces to performing $\mathcal{O}(T)$ comparisons. Since $v$ needs to be computed $\mathcal{O}(T^2)$ for all $a<b$, to ensure $v(a,t)$ and $v(t+1,b)$ are available, we obtain the complexity of $\mathcal{O}(T^3)$. The term $\mathcal{O}(EO_2)$ corresponds to computing function $f$ as well. 
\end{proof}

\subsection{Parametric scheme for solution path}\label{sec:parametric}

We now show that \eqref{eq:opt} can be solved \emph{for all values of $\bar \gamma$} with the same theoretical complexity as the dynamic programming algorithm discussed in \S\ref{sec:fixedparam}. {Observe that the optimal solution $\{\hat \theta_t\}_{t=0}^T$ of \eqref{eq:opt} is also optimal for the optimization problem 
\begin{equation}\label{eq:optCard}
	\begin{aligned}
		\min_{\{{\theta}_t\}_{t=0}^T}&\sum_{t=1}^T\left|\theta_t-\theta_{t-1}\right|^q\\
		\mathrm{s.t.}\;&\sum_{t=0}^T\mathbb{I}[\theta_t \neq 0]\leq k\\
		&\ell_t\leq\theta_t\leq u_t\qquad \forall t=0,\dots,T
	\end{aligned}
\end{equation}
with $k=\sum_{t=0}^T\mathbb{I}[\hat \theta_t \neq 0]$. Thus, to solve \eqref{eq:opt} for all values of $\bar{\gamma}$, one can instead solve \eqref{eq:optCard} for all values of $k\in \{0,\dots,T\}$ and select the solution with best objective value. Observe that this approach produces optimal solutions for all values of the regularization parameter $\bar \gamma$. Also note that since problems \eqref{eq:opt} and \eqref{eq:optCard} are non-convex, the regularized and constrained problems are not equivalent: some optimal solutions of \eqref{eq:optCard} for a given value of $k$ may not be optimal for \eqref{eq:opt} for any value of $\bar \gamma$.}

 Given $0\leq b\leq T+1$ and $k\in \mathbb{Z}$, define function $\nu(b,k)$ to be the best objective value of \eqref{eq:optCard} from indexes $0$ to $b-1$ assuming $\theta_b=0$ (if $b\leq T$), that is,
\begin{align*}
	\nu(b,k)=\min_{\{{\theta}_t\}_{t=0}^{b-1}}&\sum_{t=1}^{b-1}\left|\theta_t-\theta_{t-1}\right|^q+\mathbb{I}[b<T+1]\left|\theta_{b-1}-0\right|^q\\
	\mathrm{s.t.}\;&\sum_{t=0}^{b-1}\mathbb{I}[\theta_t \neq 0]\leq k\\
	&\ell_t\leq\theta_t\leq u_t\qquad \forall t=a,\dots,b-1,
\end{align*} 
where we adopt the convention that 	$\nu(b,k)=\infty$ if the optimization is not feasible (e.g., if $k<0$), and $\nu(0,k)=0$ for all $k\in \mathbb{Z}_+$. Our goal is to compute $\nu(T+1,k)$ for all values of $k\in \{0,\dots,T+1\}$. Note that for all $k\geq b$, we have that $\nu(b,k)=f(0,b)$, where $f$ is given by \eqref{eq:defCont}. For $k< b$, functions $\nu$ and $f$ are related through the following recursion:
\begin{lemma}
	For any $0\leq k<b\leq T+1$, we have
\begin{equation}\label{eq:recursionParam}
	\nu(b,k)=\min_{t\in\{0,\dots,b-1\}:\; 0\in [\ell_t,u_t]} \nu(t,k-(b-t-1))+f(t+1,b).
\end{equation}
\end{lemma}
\begin{proof}
	The optimal solution corresponding to $\nu(b,k)$ satisfies $\theta_t=0$ for some $0\leq t\leq b-1$ (due to the assumption that $k<b$). Letting $\bar t$ be the largest such index, and noting that the problem decomposes at the breakpoint $\bar t$ with $\theta_{\bar t}=0$, it follows that $$\nu(b,k) = \underbrace{\nu(\bar t,k-(b-\bar t-1))}_{\text{Cost from 0 to $\bar t-1$ with remaining budget}}+\underbrace{f(\bar t+1,b)}_{\text{Cost from $\bar t+1$ to $b-1$ with $\theta_t\neq 0$.}}.$$
	Therefore, $\nu(b,k)$ can be obtained by selecting the break-point $\bar t$ with least cost.
\end{proof}

\begin{lemma}
    The problem~\eqref{eq:optCard} can be solved for all values of $\bar{\gamma}$ with complexity $\mathcal{O}(T^3+EO_2)$, where $EO_2$ is the complexity of solving \eqref{eq:defCont} for all combinations of parameters $a<b$. 
\end{lemma}
\begin{proof}
Computing all values of $\nu(b,k)$ amounts to computing $f(a,b)$ for all $\mathcal{O}(T^2)$ combinations of parameters $a< b$ and doing $\mathcal{O}(T^3)$ comparisons, corresponding to $\mathcal{O}(T^2)$ choices of $b$ and $k$, with each computation of $v(b,k)$ requiring $\mathcal{O}(T)$ comparisons.
\end{proof}

Observe that if solving \eqref{eq:defCont} for all combinations of $a<b$ can be done in $\mathcal{O}(T^3)$, then the overall complexity of the method is cubic in $T$. We now show that this is indeed the case.

\subsection{Solving the subproblems}\label{sec:subproblem}

Solving \eqref{eq:opt}, either for a fixed value of $\bar \gamma$ or parametrically for all values of the regularization parameter, requires solving problem \eqref{eq:defCont} for all $0\leq a< b\leq T+1$. We now discuss how to accomplish this task efficiently, depending on the value of $q$. In particular, we show that \eqref{eq:defCont} for all $0\leq a< b\leq T+1$ can be solved in $O(T^2)$ for $q=0$, and in $O(T^3)$ for any $q\geq 1$.

\subsubsection{Case $q=0$}

We now provide an efficient algorithm (Algorithm~\ref{alg_greedy}) for solving \eqref{eq:defCont} for the case $q=0$, which we restate for convenience: \begin{equation}\label{eq:defCont0}
	\begin{aligned}
		f(a,b)=\min_{\{{\theta}_t\}_{t=a}^{b-1}}&\sum_{t=a+1}^{b-1}\mathbb{I}[\theta_t\neq\theta_{t-1}]+\mathbb{I}[a>0]\mathbb{I}[\theta_a\neq 0]+\mathbb{I}[b<T+1]\mathbb{I}[\theta_{b-1}\neq 0]\\
		\mathrm{s.t.}\;&\ell_t\leq\theta_t\leq u_t\qquad \forall t=a,\dots,b-1.
	\end{aligned} 
\end{equation}
For simplicity, we can introduce a dummy variable $\theta_{a-1}$ with bounds $\ell_{a-1}=u_{a-1}=0$ if $a>0$, and bounds $-\ell_{a-1}=u_{a-1}=\infty$ if $a=0$. Similarly, we can introduce a variable $\theta_{b}$ with bounds $\ell_{b}=u_{b}=0$ if $b<T+1$, and bounds $-\ell_{a-1}=u_{a-1}=\infty$ if $b=T+1$. With these additional variables, we can rewrite \eqref{eq:defCont0} as 
\begin{equation}\label{eq:defCont01}
	\begin{aligned}
		f(a,b)=\min_{\{{\theta}_t\}_{t=a-1}^{b}}&\sum_{t=a}^{b}\mathbb{I}[\theta_t\neq\theta_{t-1}]\\
		\mathrm{s.t.}\;&\ell_t\leq\theta_t\leq u_t\qquad \forall t=a-1,\dots,b.
	\end{aligned} 
\end{equation}
The next lemma provides a method to compute lower bounds on the objective value of \eqref{eq:defCont01}.
\begin{lemma}\label{lem:lb0}
	Given any indexes $a\leq \tau_1< \tau_2\leq b$, if $\max_{\tau_1\leq t\leq \tau_2}\ell_t>\min_{\tau_1\leq t\leq \tau_2}u_t$, then 
	\begin{equation}\label{eq:LB}\sum_{t=\tau_1+1}^{\tau_2}\mathbb{I}[\theta_t\neq\theta_{t-1}]\geq 1, \end{equation}
	for any feasible solution $\theta_t$ of \eqref{eq:defCont01}.
\end{lemma}
\begin{proof}
	The left hand side of \eqref{eq:LB} is zero if and only if $\theta_{\tau_1}=\theta_{\tau_1+1}=\dots=\theta_{\tau_2}$. However, the condition $\max_{\tau_1\leq t\leq \tau_2}\ell_t>\min_{\tau_1\leq t\leq \tau_2}u_t$ ensures that such solutions are not feasible.
\end{proof}

Algorithm~\ref{alg_greedy}  below solves \eqref{eq:defCont01} to optimality. 
At a high level, starting from time period $\tau_1=a$, the algorithm greedily finds the largest time period $\tau_2$ such that setting $\theta_{\tau_1}=\theta_{\tau_1+1}=\dots=\theta_{\tau_2}$ is feasible. In the algorithm, we maintain a set $\Gamma$ corresponding to the first and last time periods, as well as breakpoints $\tau$ such that $\theta_{\tau_1}\neq\theta_{\tau_1-1}$. We use $\texttt{last}(\Gamma)$ to denote the last element added to set $\Gamma$. 
\begin{algorithm}
	\caption{$\ell_0$ greedy}
	\label{alg_greedy}
	\begin{algorithmic}[1]
		 \State{{\bf Input:} $\hat \theta\in \R^{b-a+2}$: array to store the optimal solution}
		\State{{\bf Output:} The optimal objective value  $f(a,b)$}
		\vspace{2mm}
		\State{$\bar f\leftarrow 0$}
		\State{$\bar \ell\leftarrow\ell_{a-1}$, $\bar u\leftarrow u_{a-1}$} \Comment{Current bounds on constant interval}
		\State{$\Gamma\leftarrow\{a-1\}$} \Comment{Set of breakpoints}\label{line:setGamma}
		\For{$t=a,\dots,b$}
		\If{$\max\{\ell_{t},\bar \ell\}<\min\{u_{t},\bar u\}$}  \label{line:if}\Comment{Possible to extend constant interval}
	\State{$\bar \ell\leftarrow\max\{\ell_{t},\bar \ell\}$, $\bar u\leftarrow \min\{u_{t},\bar u\}$}
		\Else \Comment{Breakpoint found}
		\State $\bar f\leftarrow \bar f+1$ \Comment{Updates objective}
		\State $\hat \theta_\tau \leftarrow (\bar \ell+\bar u)/2\qquad \forall \tau=\texttt{last}(\Gamma),\dots,t-1$ \Comment{Updates solution}
		\State{$\Gamma\leftarrow\Gamma\cup \{t\}$}\label{line:increaseGamma}
		\State{$\bar \ell\leftarrow\ell_{t}$, $\bar u\leftarrow u_{t}$} \Comment{Reset bounds on interval}
		\EndIf
		\EndFor
		\If{$\texttt{last}(\Gamma)\neq b$} \Comment{Either $b=T+1$ or $0\in [\bar \ell,\bar u]$}
		\State $\hat \theta_\tau \leftarrow (\bar \ell+\bar u)/2\qquad \forall \tau=\texttt{last}(\Gamma),\dots,b$
		\State{$\Gamma\leftarrow\Gamma\cup \{b\}$}\label{line:conditionalIncrease}
		\EndIf
	
		\State \textbf{Return }{$\bar f$;}
	\end{algorithmic}
\end{algorithm}

\begin{proposition}\label{prop_beta1}
Algorithm~\ref{alg_greedy} runs in $\mathcal{O}(b-a)$ time and memory, and returns a solution $\{\hat \theta_t\}_{a-1}^{b}$ that is optimal for \eqref{eq:defCont01}. Moreover, given any $b_0<b$, the truncated solution $\{\hat \theta_t\}_{a-1}^{b_0}$ is optimal for the problem associated with function $f(a,b_0)$.
\end{proposition}
\begin{proof}
	The runtime readily follows by counting the number of iterations in Algorithm~\ref{alg_greedy}. We first prove the first sentence of the proposition. Let $a-1=\tau_0<\tau_1<\tau_2<\ldots<\tau_m=b$ be the elements of the set $\Gamma$ from Algorithm~\ref{alg_greedy}. By construction, $\max_{\tau_i\leq t\leq \tau_{i+1}}\ell_t>\min_{\tau_i\leq t\leq \tau_{i+1}}u_t$ for all $i=0,\dots,m-2$ (since elements are added to $\Gamma$ in line~\ref{line:increaseGamma} only when the condition in line~\ref{line:if} fails). It follows from Lemma~\ref{lem:lb0} that 
	$$m-1\leq\sum_{i=0}^{m-2}\sum_{t=\tau_i+1}^{\tau_{i+1}}\mathbb{I}[\theta_t\neq \theta_{t-1}]=\sum_{t=a-1}^{\tau_{m-1}}\mathbb{I}[\theta_t\neq \theta_{t-1}].$$
	Moreover, if $\ell_b=u_b=0$ 
 and $0\not\in [\max_{\tau_{m-1}\leq t\leq b-1}\ell_t,\min_{\tau_{m-1}\leq t\leq b-1}u_t]$, 
  then 
	$$1\leq \sum_{t=\tau_{m-1}+1}^{b}\mathbb{I}[\theta_t\neq \theta_{t-1}]$$
	since $\theta_t$ needs to be nonzero for some $\tau_{m-1}\leq t\leq \tau_{b-1}$ and $\theta_b=0$. Therefore, we find the lower bound on the optimal objective value of \eqref{eq:defCont01} given by 
	$$m-1+\mathbb{I}\left[\text{$\ell_b=u_b=0$ and $0\not\in [\max_{\tau_{m-1}\leq t\leq b-1}\ell_t,\min_{\tau_{m-1}\leq t\leq b-1}u_t]$}\right]\leq\sum_{t=a-1}^{b}\mathbb{I}[\theta_t\neq \theta_{t-1}].$$
	We now check the objective value of Algorithm~\ref{alg_greedy}, given by $\bar f$. Note that set $\bar \Gamma$ contains $m+1$ elements. The first element added (line~\ref{line:setGamma}) does not increase the objective value. The next $m-1$ elements added (line~\ref{line:increaseGamma}) increase the objective function by one. Finally, the last element added increases the objective value by one unless $b=T+1$ (in which case $-\ell_b=u_b=\infty$) or $0\in [\bar \ell,\bar u]$ (line~\ref{line:conditionalIncrease}). Thus, the value $\bar f$ returned by Algorithm~\ref{alg_greedy} matches the lower, and is thus optimal.
	
	The second sentence of the proposition follows since Algorithm~\ref{alg_greedy} is greedy: the first $b_0$ steps used when computed $f(a,b)$ are exactly the same as those used to compute $f(a,b_0)$. 
\end{proof}
 From Proposition~\ref{prop_beta1} we see that solving \eqref{eq:defCont0} for a fixed $a$ and all $b>a$ can be done in linear time. It follows that solving \eqref{eq:defCont0} for all $a<b$ is possible in $\mathcal{O}(T^2)$ time. 

\subsubsection{Case $q\geq 1$}
{
We now discuss how to solve \eqref{eq:defCont} for $q\geq 1$.
We assume for simplicity that $0<a<b<T+1$, although the arguments can easily be extended to the cases with $a=0$ or $b=T+1$.

\begin{lemma}\label{lem:interpolation}
If $q\geq 1$, there exists an optimal solution $\theta^*$ of \eqref{eq:defCont} where: $\bullet$ for all $a<t<b$, either $\theta_t^*=(\theta_{t-1}^*+\theta_{t+1}^*)/2$, or $\theta_t^*\in \{\ell_t, u_t\}$; $\bullet$ either $\theta_a^*=\theta_{a+1}^*$, or $\theta_a^*\in \{\ell_a, u_a\}$; $\bullet$ either $\theta_{b-1}^*=\theta_{b-2}^*$, or $\theta_{b-1}^*\in \{\ell_{b-1}, u_{b-1}\}.$
\end{lemma}
\begin{proof}
We only prove the first statement in the lemma, since the other two follow from identical arguments. Given any index $a<\tau<b-1$, consider optimization \eqref{eq:defCont} where all values except for $\theta_\tau$ are fixed to any feasible value (we omit terms that do not depend on $\theta_\tau$)
\begin{equation}\label{eq:defContTau}
\begin{aligned}
	\min_{\theta_\tau}&\left|\theta_{\tau}-\theta_{\tau-1}\right|^q+\left|\theta_{\tau+1}-\theta_{\tau}\right|^q\\
	\mathrm{s.t.}\;&\ell_\tau\leq\theta_\tau\leq u_\tau.
\end{aligned} 
\end{equation}
Observe that if an optimal solution $\theta_\tau^*$ of \eqref{eq:defContTau} satisfies $\theta_\tau^*> \max\{\theta_{\tau-1},\theta_{\tau+1}\}$, then necessarily $\theta_\tau^*=\ell_\tau$, since otherwise it would be possible to decrease $\theta_\tau^*$, improving the objective value. Similarly the case $\theta_\tau^*<\min\{\theta_{\tau-1},\theta_{\tau+1}\}$ implies that $\theta_\tau^*=u_\tau$. 

To prove the first statement, assume that $\theta_\tau^*\not\in\{\ell_t,u_t\}$, which can only happen if $\min\{\theta_{\tau-1},\theta_{\tau+1}\}\leq \theta_\tau^*\leq \max\{\theta_{\tau-1},\theta_{\tau+1}\}$. If $\theta_{\tau-1}\leq \theta_{\tau+1}$, then it follows that $\theta_\tau$ is optimal for 
\begin{align}\label{eq_single0}
    \min_{\theta_\tau\in \R}\left(\theta_{\tau}-\theta_{\tau-1}\right)^q+\left(\theta_{\tau+1}-\theta_{\tau}\right)^q.
\end{align}
First, suppose that $q>1$. Then, by convexity and differentiability, the derivative of~\eqref{eq_single0} at the optimal solution should be zero. This implies
$$q(\theta_t-\theta_{t-1})^{q-1}-q(\theta_{t+1}-\theta_t)^{q-1}=0\Leftrightarrow \theta_t-\theta_{t-1}=\theta_{t+1}-\theta_t\Leftrightarrow \theta_\tau=(\theta_{t+1}+\theta_{t-1})/2.$$
Moreover, if $q=1$, it can be easily seen that any feasible value $\theta_\tau\in [\theta_{\tau-1},\theta_{\tau+1}]$ is optimal. Thus, at least one of the following five cases holds: \textit{(i)} $\theta_\tau^*=(\theta_{t+1}+\theta_{t-1})/2\in [\theta_{\tau-1},\theta_{\tau+1}]$ is optimal; \textit{(ii)} $\theta_\tau^*=\ell_t\in [\theta_{\tau-1},\theta_{\tau+1}]$ is optimal; \textit{(iii)} $\theta_\tau^*=u_t\in [\theta_{\tau-1},\theta_{\tau+1}]$;  \textit{(iv)} $\theta_\tau^*=\ell_t\geq \theta_{\tau+1}$ is optimal; or \textit{(v)} $\theta_\tau^*=u_t\leq \theta_{\tau-1}$ is optimal. Since all fives cases satisfy $\theta_\tau^*\in \{\ell_\tau, (\theta_{t+1}+\theta_{t-1})/2,u_\tau\}$, the statement of the lemma holds.

The case $\theta_{\tau-1}\geq \theta_{\tau+1}$ is handled identically, concluding the proof. 
\end{proof}

From Lemma~\ref{lem:interpolation}, we find that $\theta_t$ (when not fixed to a bound) is obtained by interpolation of the adjacent values. In other words, in an optimal solution to \eqref{eq:defCont}, there is a set of indexes $a\leq \tau_1<\tau_2<\dots<\tau_m\leq b-1$ corresponding to variables fixed to a bound, and for a variable $t$ who is not at a bound with $\tau_j<t<\tau_{j+1}$, we find its optimal value given by \begin{equation}\label{eq:interpolation}\theta_t^*=\theta_{\tau_j}^*+\frac{t-\tau_j}{\tau_{j+1}-\tau_j}(\theta_{\tau_{j+1}}^*-\theta_{\tau_j}^*).\end{equation} 
Moreover, $\theta_1>a$ (thus $\theta_{a}^*\not\in \{\ell_a,u_a\}$), then it also follows from Lemma~\ref{lem:interpolation} that $\theta_{a}^*=\theta_{a_+1}^*=\theta_{\tau_1}$; similarly, if $\tau_m<b-1$, then $\theta_{\tau_m}^*=\theta_{\tau_m+1}^*=\theta_{b-1}$.

We now propose a dynamic programming approach for solving \eqref{eq:defCont} with $q\geq 1$.
Let set $S=\left\{\texttt{lb},\texttt{mid},\texttt{ub}\right\}$, where $\texttt{lb}$ stands for ``lower bound", $\texttt{ub}$ stands for ``upper bound", and $\texttt{mid}$ stands for ``neither lower or upper bound". Given $s_a,s_{b}\in S$, define $f^0(a,s_a,b,s_b)$ as the objective value of \eqref{eq:defCont} when $\theta_a$ and $\theta_{b-1}$ are set to their bounds indicated by $s_a$ and $s_b$, respectively, and the remaining variables are set according to \eqref{eq:interpolation}. Note that $f^0(a,s_a,b,s_b)$ can be computed in $\mathcal{O}(T)$ time, and $f^0(a,s_a,b,s_b)=\infty$ if the sequence indicated by \eqref{eq:interpolation} is infeasible. 

Given $s_a,s_b\in S$, define $f^1(a,s_a,b,s_b)$ as the optimal solution of \eqref{eq:defCont}, and observe that $f^1$ can be computed via the recursion (where $a\leq t<b-1$ is the largest index of a variable set to a bound)
\begin{equation}\label{eq:recursion2}
f^1(a,s_a,b,s_b)=\min_{t\in\{a,\dots,b-1\},s\in \{\texttt{lb},\texttt{ub}\}}f^1(a,s_1,t+1,s)+f^0(t,s,b,s_b).
\end{equation}
Thus, assuming that values $f^1(a,s_a,t,s)$ with $s\in S$ and $t<b$ have been computed, \eqref{eq:recursion2} can be computed in $\mathcal{O}(T)$ time, and computing $f^1$ for all $\mathcal{O}(T^2)$ combinations of parameters requires $\mathcal{O}(T^3)$ time. Figure~\ref{fig:enter-label} illustrates an example of the recursion~\eqref{eq:recursion2}.

	\section{Statistical Guarantees}\label{sec_stats}

 \begin{figure}
    \begin{center}
    \subfloat[\scriptsize$f^0(0,\texttt{ub},5,\texttt{lb})=\infty$]{\includegraphics[width=0.45\textwidth,trim={6cm 3cm 12cm 3cm},clip]{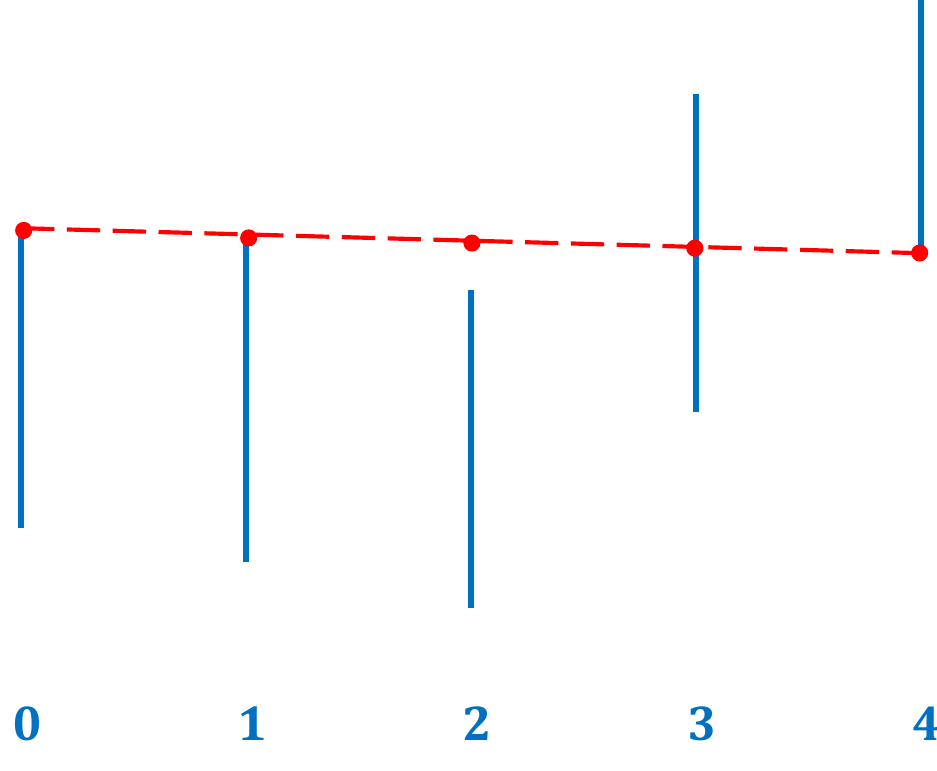}}\hfill \subfloat[\scriptsize${f^1(0,\texttt{mid},5,\texttt{lb})=f^1(0,\texttt{mid},3,\texttt{ub})+f^0(2,\texttt{ub},5,\texttt{lb})}$]{\includegraphics[width=0.45\textwidth,trim={6cm 3cm 12cm 3cm},clip]{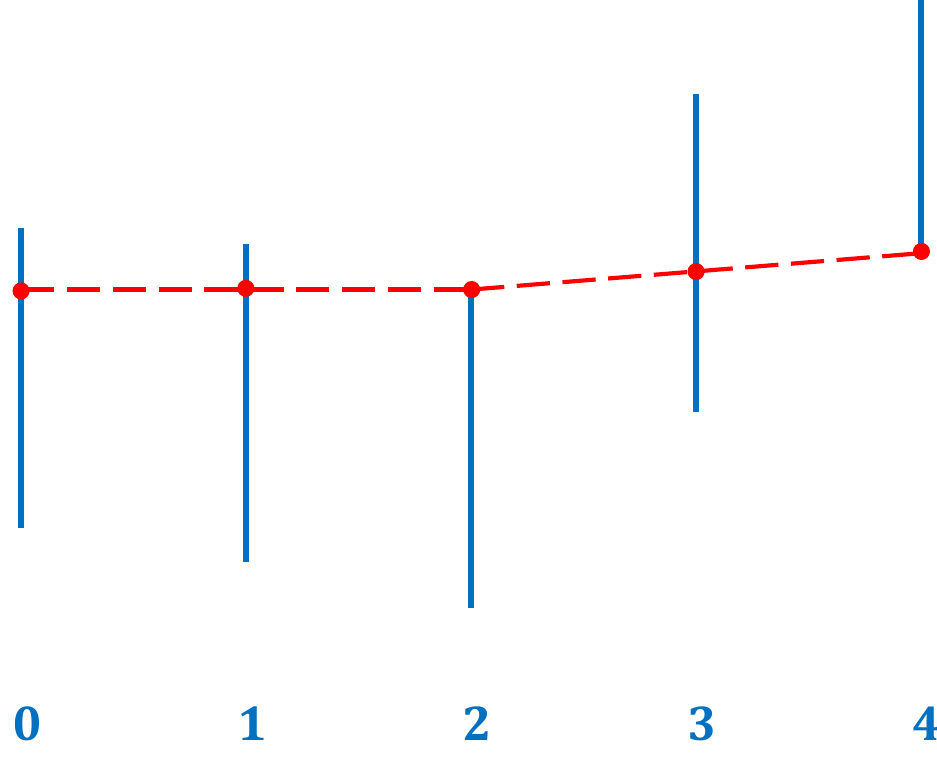}}
    \end{center}
    \caption{Example of recursion~\eqref{eq:recursion2}. (Left) Evaluation of the cost setting $\theta_0=u_0$, $\theta_4=\ell_4$ and remaining variables by interpolation \eqref{eq:interpolation}. Since this solution is infeasible due to $\theta_2\not\in [\ell_2,u_2]$, the cost is infinity.
    (Right) Optimal solution, which sets $\theta_2=u_2$. }
    \label{fig:enter-label}
\end{figure}
	In the previous section, we showed that~\ref{generalopt1} can be solved efficiently for all values of the regularization parameter $0\leq \gamma\leq 1$. However, the efficient solvability of~\ref{generalopt1} is contingent upon the availability of the approximate backward mapping $\widetilde{F}^*$. On the other hand, according to Theorem~\ref{thm_constrained}, the accuracy of this backward mapping and the choice of $\lambda_t$ directly controls the estimation error of the obtained parameters. 
In this section, our goal is to provide statistically and computationally efficient ways to obtain approximate backward mappings for two classes of time-varying MRFs, namely time-varying GMRFs and time-varying DMRFs. We then use our results to provide end-to-end sample complexity guarantees for the solution of~\ref{generalopt1}.

 To achieve this goal, we first present a general statistical bound on the backward mapping deviation. 
 Recall that $\mathcal{R}$ is the domain of the backward mapping and $\mathcal{B}_{\infty}(\mu^\star, \alpha)$ is the $\ell_\infty$-ball with radius $\rho$ centered at $\mu^\star$.
	\begin{sloppypar}
		\begin{definition}[Local lipschitzness of $\pmb{\widetilde{F}^*}$]\label{asp_lip}
			$\widetilde{F}^*$ is $(L,\alpha)$-locally Lipschitz for some $L,\alpha>0$ if $\|\widetilde{F}^*(\mu) - \widetilde{F}^*(\mu')\|_\infty\leq L\|\mu-\mu'\|_\infty$ for every $\mu,\mu'\in\mathcal{R}\cap\mathcal{B}_{\infty}(\mu^\star, \alpha)$.
		\end{definition}
	\end{sloppypar}
The local Lipschitzness of the approximate backward mapping will play a crucial role in our subsequent statistical analysis. Roughly speaking, it entails that if the mean parameters $\mu, \mu'\in\mathcal{R}\cap\mathcal{B}^{\infty}_{\mu^\star}(\alpha)$ are close, so are their images under $\widetilde{F}^*$. 

\begin{sloppypar}
	\begin{assumption}[Exponential tail bound]\label{asp_exp}
		Let $\widehat{\mu}_t =  \frac{1}{N_t}\sum_{i=1}^{N_t}\phi\left(x^{(i)}_t\right)$ be the empirical mean parameter and $\mu^\star_t = \mathbb{E}_{\theta_t^\star}[\phi(x_t)]$ its true counterpart. There exist $c, c', v>0$ such that
		\begin{align*}
			\mathbb{P}(|\widehat{\mu}_{t;i} - {\mu}^\star_{t;i}|\geq \delta)\leq c'\exp(-cN_t\delta^2)\quad \text{for every}\quad \delta\in(0,v],\ \ i=1,\dots,p,\ \ t=0,\dots, T
		\end{align*}
	\end{assumption}
\end{sloppypar}
The above assumption implies that the empirical mean parameters are concentrated around their expectation with a tail bound that decreases exponentially fast. As will be shown later, this assumption is satisfied for a wide range of MRFs, including GMRFs and DMRFs.

Recall that $\norm{\theta_t^\star-\widetilde{F}^*(\widehat{\mu}_t)}_\infty$ plays an important role in Theorem~\ref{thm_constrained}. 
Under the local Lipschitzness of $\widetilde{F}^*$ and Assumption~\ref{asp_exp}, we will next show that the backward mapping deviation $\norm{\theta_t^\star-\widetilde{F}^*(\widehat{\mu}_t)}_\infty$ can be controlled with high probability.
\begin{proposition}[Backward mapping deviation]\label{lem_back_dev}
	Suppose that Assumption~\ref{asp_exp} is satisfied and $\widetilde{F}^*$ is $(L,\alpha)$-locally Lipschitz with some $\alpha, L>0$. Then, for any constant $\tau\geq 0$ and sample size $N_t\geq \frac{1+\tau}{c\min\{v^2, \alpha^2\}}\cdot \log p$, we have
	\begin{align*}
		\norm{\theta_t^\star-\widetilde{F}^*(\widehat{\mu}_t)}_\infty\leq \underbrace{\norm{\theta_t^\star-\widetilde{F}^*({\mu}^\star_t)}_\infty}_{:=\Delta_t}+\sqrt{\frac{(1+\tau)L^2}{c}\cdot\frac{\log p}{N_t}}\!\!\!\!,\quad \text{with probability of } 1-p^{-\tau}.
	\end{align*}
\end{proposition}
\begin{proof}
Applying Assumption~\ref{asp_exp} with the choice of $\delta = \sqrt{\frac{1+\tau}{c}\cdot \frac{\log p}{N_t}}$ implies that
\begin{align*}
    \mathbb{P}\left(|\widehat{\mu}_{t;i} - {\mu}^\star_{t;i}|\leq \sqrt{\frac{1+\tau}{c}\cdot \frac{\log p}{N_t}}\right)\geq 1-c'p^{-(1+\tau)},
\end{align*}
where $\delta\leq v$ is automatically satisfied due to the lower bound on $N_t$. Then, a simple union bound over index $i$ leads to
	\[
	\norm{\widehat{\mu}_t-\mu^\star_t}_\infty\leq \sqrt{\frac{1+\tau}{c}\cdot\frac{\log p}{N_t}}\leq \alpha, \quad \text{with probability of $1-c'p^{-\tau}$.}
	\]
	Under this event, one can write
	\begin{align*}
		\norm{\theta_t^\star-\widetilde{F}^*(\widehat{\mu}_t)}_\infty &\leq \norm{\theta_t^\star-\widetilde{F}^*({\mu}^\star_t)}_\infty+\norm{\widetilde{F}^*(\widehat{\mu}_t)-\widetilde{F}^*({\mu}^\star_t)}_\infty\\
		&\leq \norm{\theta_t^\star-\widetilde{F}^*({\mu}^\star_t)}_\infty+L\norm{\widehat{\mu}_t-{\mu}^\star_t}_\infty\\
		&\leq \norm{\theta_t^\star-\widetilde{F}^*({\mu}^\star_t)}_\infty+\sqrt{\frac{(1+\tau)L^2}{c}\cdot\frac{\log p}{N_t}}
	\end{align*}
where in the second inequality we used the $(L,\alpha)$-local Lipschitzness of $\widehat{F}^*$.
\end{proof}
Proposition~\ref{lem_back_dev} shows that in order to guarantee a small backward mapping deviation, $\widetilde{F}^*$ should satisfy two properties: first, it must incur a small error at the true mean parameter (captured by $\Delta_t$), and second, it must be $(\alpha, L)$-Lipschitz for some $\alpha$ and $L$. However, there is often a trade-off between these two properties. To see this trade-off, consider $\widetilde{F}^*(\Sigma) = \Sigma^{-1}$ as a natural choice of the approximate backward mapping for GMRFs. Evidently, we have $\Delta_t = 0$ for $\widetilde{F}^*(\Sigma) = \Sigma^{-1}$. However, this choice of approximate backward mapping loses its local Lipschitzness as $\Sigma$ approaches singularity (i.e., when the sample size falls below the dimension). To alleviate this issue, we propose a \textit{proxy} backward mapping that strikes a balance between $\Delta_t$ and the local Lipschitzness.

	\subsection{Approximate Backward Mapping for GMRFs} 
	Suppose that at any given time $t$, a sequence of data samples $\left\{x_t^{(i)}\right\}_{i=1}^{N_t}\subset \mathbb{R}^n$ are collected from a time-varying GMRF with distribution~\eqref{gmrf}. Our next lemma is borrowed from~\citep{ravikumar2011high} and shows that the empirical mean parameters (i.e., sample covariance matrices) satisfy Assumption~\ref{asp_exp}.
	\begin{lemma}[Lemma 1,~\cite{ravikumar2011high}]\label{l_Sigma}
		Suppose the samples $\left\{x_t^{(i)}\right\}_{i=1}^{N_t}$ are collected from a time-varying GMRF. Then, the empirical mean parameter defined as $\widehat{\Sigma}_t = ({1}/{N_t})\sum_{i=1}^{N_t}x_t^{(i)}{x_t^{(i)}}^\top$ satisfies Assumption~\ref{asp_exp} with $c'=4$, $c = 3200\max_i\{\Sigma^\star_{t;ii}\}^2$, and $v = 40\max_i\{\Sigma^\star_{t;ii}\}$.
	\end{lemma}
	 Recall that the canonical parameter corresponds to the inverse covariance matrix for GMRFs. Hence, we have $p = n^2$. To address the singularity of the backward mapping when $N_t< n$, we use the approximate backward mapping introduced by~\citet{yang2014elementary}: consider the soft-thresholding operator $\mathrm{ST}_{\nu}:\mathbb{R}^{n\times n}\to \mathbb{R}^{n\times n}$ defined as $\texttt{ST}_{\nu}(M)_{ij} = M_{ij}-\mathrm{sign}(M_{ij})\min\{|M_{ij}|,\nu\}$ if $i\not=j$, and $\texttt{ST}_{\nu}(M)_{ij} = M_{ij}$ if $i=j$. The proposed approximate backward mapping is then given by $\widetilde F^*(\widehat \Sigma_t) = [\texttt{ST}_{\nu}(\widehat \Sigma_t)]^{-1}$. Clearly, we have $\Delta_t = 0$ for the choice of $\nu=0$. However, we will later show that $\nu>0$ is essential for guaranteeing the $(L,\alpha)$-local Lipschitzness of $\widetilde F^*$  when $N_t< n$. 
	
	\begin{assumption}[Bounded norm]\label{asp_bound}
		There exist constant numbers $\kappa_1<\infty$, $\kappa_2>0$, and $\kappa_3<\infty$ such that
		\begin{align*}
			\|\Theta^\star_t\|_\infty\leq \kappa_1,\; \inf_{w:\|w\|_\infty=1}\|\Sigma^\star_tw\|_\infty\geq \kappa_2,\; \|\Sigma^\star_t\|_{\linf}\leq \kappa_3,\quad \text{for every $t = 0,\dots, T$.}
		\end{align*}
	\end{assumption}
The above assumption is fairly mild and implies that the true covariance matrices and their inverses have bounded norms. Without loss of generality and to streamline the presentation, we assume that $\kappa_1,\kappa_3\geq 1$ and $\kappa_2\leq 1$.

\begin{definition}[Weak sparsity]\label{def_weak}
	We say $\Sigma \in\mathbb{R}^{n\times n}$ is $(s,r)$-weakly sparse for some $s\geq 0$ and $0\leq r<1$ if $\max_{i}\sum_{j=1}^n|\Sigma_{ij}|^r\leq s_r$.
\end{definition} 
Intuitively, we say that ``the true covariance matrices are weakly sparse'' if $\{\Sigma^\star_t\}_{t=0}^T$ are $(s,r)$-weakly sparse with $s_r\ll n$ for some $0\leq r< 1$. 
The notion of weak sparsity extends the classical notion of sparsity for matrices. 
Indeed, except for a few special cases, a sparse matrix does not have a sparse inverse. Consequently, a sparse precision matrix may {not} lead to a sparse covariance matrix. However, a large class of sparse precision matrices have weakly sparse inverses:
\begin{sloppypar}
	\begin{itemize}
		\item If $\Theta^\star_t$ has a banded structure with small bandwidth, then it is known that the elements of $\Sigma^\star_t = \Theta_t^{-1}$ enjoy exponential decay away from the main diagonal~\citep{demko1984decay, kershaw1970inequalities}. Under such condition, one can verify that $s\leq\frac{C}{1-\rho^r}$ for some constant $C>0$ and $\rho<1$.
		\item A similar statement holds for a class of inverse covariance matrices whose support graphs have large average path length~\citep{benzi2007decay, benzi2015decay}; a large class of inverse covariance matrices with row- and column-sparse structures satisfy this condition.
		\item Suppose that $\Theta^\star_t$ is diagonally dominant with $\Theta^\star_{t,ii}-\sum_{j\not=i}|\Theta^\star_{t,ij}| > 0$ for $i=1,2,\dots,d$. Then, a simple derivation reveals that $\Sigma^\star_t$ is $(s,r)$-weakly sparse with $s\leq d^{1-r}/\left(\min_i \left\{\Theta^\star_{t,ii}-\sum_{j\not=i}|\Theta^\star_{t,ij}|\right\}\right)^{r}$ for any $0\leq r< 1$. The diagonally dominant structures naturally arise in the context of graphical model inference; see~\citep{egilmez2016graph} for various applications, including those with graph Laplacian structures.
	\end{itemize}
\end{sloppypar}
Under Assumption~\ref{asp_bound} and $(s,r)$-weak sparsity of the true covariance matrices, we show that the proposed approximate backward mapping $\widetilde F^*(\widehat \Sigma_t) = [\texttt{ST}_{\nu}(\widehat \Sigma_t)]^{-1}$ is $(L,\alpha)$-locally Lipschitz and has small $\Delta_t$.
\begin{sloppypar}
	\begin{lemma}\label{lem_error_lip}
		Suppose that Assumption~\ref{asp_bound} is satisfied and $\Sigma^\star_t$ is $(s,r)$-weakly sparse for some $s\geq 0$ and $0\leq r<1$.
		Then, the following statements hold for the approximate backward mapping $\widetilde F^*(\widehat \Sigma_t) = [\texttt{ST}_{\nu}(\widehat \Sigma_t)]^{-1}$:
		\begin{itemize}
			\item We have $\Delta_t\leq({2\kappa_1}/{\kappa_2})\nu$;
			\item Assume that $\nu\leq (\kappa_2/(20s))^{1/(1-r)}$. Then, $\widetilde F^*(\Sigma)$ is $(L,\alpha)$-locally Lipschitz with $L = 4/\kappa_2^2$ and $\alpha = \min\{\kappa_2\nu^r/(24s), \nu/2\}$.
		\end{itemize}
	\end{lemma}
\end{sloppypar}
To provide the proof of the above lemma, we will rely on the following intermediate lemma adapted from~\citep[Lemma 1]{yang2014elementary}.
\begin{lemma}[\cite{yang2014elementary}]\label{lem_yang}
	    We have 
	    \begin{align*}
		\norm{\texttt{ST}_{\nu}(\Sigma)-\Sigma^\star_t}_\infty\leq 5s_r\nu^{1-r}+3s_r\nu^{-r}\alpha,
	\end{align*}
	for every $\Sigma\in \mathcal{B}_{\Sigma_t^\star}(\alpha)$ and $\alpha\leq \nu/2$.
	\end{lemma}
\paragraph{Proof of Lemma~\ref{lem_error_lip}}
		We first start with the proof of the first statement. One can write
		\begin{align}
			\Delta_t =  \norm{\Theta_t^\star - [\texttt{ST}_{\nu}(\Sigma_t^\star)]^{-1}}_{\infty/\infty}&\leq \norm{\Theta^\star_t}_\infty\norm{[\texttt{ST}_{\nu}(\Sigma_t^\star)]^{-1}}_\infty\norm{\texttt{ST}_{\nu}(\Sigma_t^\star)-\Sigma_t^\star}_{\infty/\infty}\nonumber\\
			&\leq \kappa_1\nu\norm{[\texttt{ST}_{\nu}(\Sigma_t^\star)]^{-1}}_\infty\label{eq_deltat}
		\end{align}
	where the last inequality follows from $\norm{\texttt{ST}_{\nu}(\Sigma_t^\star)-\Sigma_t^\star}_{\infty/\infty}\leq \nu$ and our assumption $\norm{\Theta^\star_t}_\infty\leq \kappa_1$. Next, we provide an upper bound for $\norm{[\texttt{ST}_{\nu}(\Sigma_t^\star)]^{-1}}_\infty$. 
	One can write
	\begin{align*}
		\norm{\texttt{ST}_{\nu}(\Sigma_t^\star)w}_\infty&\geq \norm{\Sigma_t^\star w}_\infty-\norm{\left(\texttt{ST}_{\nu}(\Sigma_t^\star)-\Sigma_t^\star\right)w}_\infty\\
		&\geq \left(\kappa_2-\norm{\left(\texttt{ST}_{\nu}(\Sigma_t^\star)-\Sigma_t^\star\right)}_\infty\right)\norm{w}_\infty\\
		&\geq (\kappa_2/2)\norm{w}_\infty
	\end{align*}
	where the last inequality is due to our choice of $\nu$. This implies that $\norm{[\texttt{ST}_{\nu}(\Sigma_t^\star)]^{-1}}_\infty\leq 2/\kappa_2$. Combining this inequality with~\eqref{eq_deltat} leads to $\Delta_t\leq (2\kappa_1/\kappa_2)\nu$.
	
	Moving on to the proof of the second statement, we invoke Lemma~\ref{lem_yang} to write
	\begin{align*}
		\norm{\texttt{ST}_{\nu}(\Sigma)-\Sigma^\star_t}_\infty\leq 5s_r\nu^{1-r}+3s_r\nu^{-r}\alpha \qquad \text{for every $\Sigma\in \mathcal{B}_{\Sigma_t^\star}(\alpha)$ and $\alpha\leq \nu/2$.}
	\end{align*}
Based on our assumption on $\nu$, we have $5s_r\nu^{1-r}\leq \kappa_2/4$. Similarly, based on our assumed upper bound on $\alpha$, we have $3s_r\nu^{-r}\alpha\leq \kappa_2/4$. This implies that $\norm{\texttt{ST}_{\nu}(\Sigma)-\Sigma^\star_t}_\infty\leq \kappa_2/2$ for every $\Sigma\in \mathcal{B}_{\Sigma_t^\star}(\alpha)$. On the other hand, for every $\Sigma, \Sigma'\in \mathcal{B}_{\Sigma_t^\star}(\alpha)$, we have 
\begin{align}\label{eq_ub}
	\norm{[\texttt{ST}_{\nu}(\Sigma)]^{-1}-[\texttt{ST}_{\nu}(\Sigma')]^{-1}}_{\infty/\infty}\leq\norm{[\texttt{ST}_{\nu}(\Sigma)]^{-1}}_\infty\norm{[\texttt{ST}_{\nu}(\Sigma')]^{-1}}_\infty\norm{\Sigma-\Sigma'}_{\infty/\infty}
\end{align}
Similarly, one can write
\begin{align*}
	\|\texttt{ST}_{\nu}({\Sigma})w\|_\infty&\geq \|\Sigma^\star_t w\|_\infty-\left\|(\texttt{ST}_{\nu}({\Sigma})-\Sigma^\star_t)w\right\|_\infty\nonumber\\
	&\geq\left(\kappa_2-\left\|\texttt{ST}_{\nu}({\Sigma})-\Sigma^\star_t\right\|_\infty\right)\|w\|_\infty\\
	&\geq (\kappa_2/2)\|w\|_\infty
\end{align*}
where in the last inequality, we used $\norm{\texttt{ST}_{\nu}(\Sigma)-\Sigma^\star_t}_\infty\leq \kappa_2/2$. This in turn implies $\norm{[\texttt{ST}_{\nu}(\Sigma)]^{-1}}_\infty\leq 2/\kappa_2$. Similarly, we have $\norm{[\texttt{ST}_{\nu}(\Sigma')]^{-1}}_\infty\leq 2/\kappa_2$. Combining these inequalities with~\eqref{eq_ub} leads to 
\begin{align*}
	\norm{[\texttt{ST}_{\nu}(\Sigma)]^{-1}-[\texttt{ST}_{\nu}(\Sigma')]^{-1}}_{\infty/\infty}\leq(4/\kappa_2^2)\norm{\Sigma-\Sigma'}_{\infty/\infty},
	 \qquad \text{for every $\Sigma, \Sigma'\in \mathcal{B}_{\Sigma_t^\star}(\alpha)$.}
\end{align*}
This completes the proof. $\hfill\square$

Lemma~\ref{lem_error_lip} highlights the trade-off in choosing the threshold $\nu$: to keep the error $\Delta_t$ small, it is desirable to pick a small value for $\nu$. However, $\nu$ cannot be too small, as it would shrink the radius of the neighborhood within which $\widetilde{F}^*$ is locally Lipschitz. Our next lemma shows that it suffices to pick $\nu$ large enough so that $\norm{\widehat{\Sigma}_t-\Sigma_t^\star}_{\infty/\infty}\leq \alpha$ with high probability, where $\alpha$ is a function of $\nu$ defined in Lemma~\ref{lem_error_lip}. To streamline the presentation, we first define the following quantities which will be used extensively in our subsequent arguments:
\begin{align*}
    &\mathcal{N}_{\mathrm{g}} := \left(\frac{4s_r}{\kappa_3\kappa_2}\right)^{\frac{2}{1-r}}, \qquad \mathcal{E}_{\mathrm{g}} := \frac{10\kappa_1+2\kappa_3^{-1}}{5\kappa_2^2},\\
    \bar\Theta_t :=\!\! \min_{(i,j)\in\mathcal{S}_t}|\Theta^\star_{t;ij}|,\ &\Delta\bar\Theta_t :=\!\! \min_{(i,j)\in\mathcal{D}_t}|\Theta^\star_{t;ij}-\Theta^\star_{t-1;ij}|,\  \Theta_t^{\min} := \begin{cases}
    \min\left\{\bar\Theta_t, \frac{\Delta\bar\Theta_t}{2}, \frac{\Delta\bar\Theta_{t-1}}{2}\right\} & t\geq 1\\
    \min\left\{\bar\Theta_t, \frac{\Delta\bar\Theta_t}{2}\right\} & t=0
    \end{cases}
\end{align*}
The quantities $\mathcal{N}_{\mathrm{g}}$ and $\mathcal{E}_{\mathrm{g}}$ will be used in our subsequent error bounds, while $\bar\Theta_t$, $\Delta\bar\Theta_t$, and $\Theta_t^{\min}$ are defined to capture the minimum nonzero elements of the true parameter and their differences.

\begin{lemma}\label{lem_bound_gmrf}
	Suppose that Assumption~\ref{asp_bound} is satisfied and $\Sigma^\star_t$ is $(s,r)$-weakly sparse for some $s\geq 0$ and $0\leq r<1$. Suppose that $N_t\geq \mathcal{N}_{\mathrm{g}}\log n$. Then, with the choice of $\nu_t = \sqrt{\log n/N_t}$, we have
	\[
	\norm{\Theta_t^\star-[\texttt{ST}_{\nu_t}(\widehat{\Sigma}_t)]^{-1}}_{\infty/\infty}\leq \mathcal{E}_{\mathrm{g}}\sqrt{\frac{\log n}{N_t}},\qquad \text{with probability of $1-4n^{-15}$.}
	\]
\end{lemma}
\begin{sloppypar}
    \begin{proof}
	In order to apply Proposition~\ref{lem_back_dev}, we need to verify that $N_t\geq (1+\tau)/(c\min\{v^2,\alpha^2\})\log n$. This is readily implied by setting $\tau = 15$, and invoking the derived values for $c$, $v$, and $\alpha$ in Lemmas~\ref{l_Sigma} and~\ref{lem_error_lip} together with our choice of $\nu_t$ and the assumed lower bound on $N_t$. The result then follows from Proposition~\ref{lem_back_dev} and the provided upper bounds on $\Delta_t$ and $L$ in Lemma~\ref{lem_error_lip}.
\end{proof}
\end{sloppypar}
Given Lemma~\ref{lem_bound_gmrf}, we are ready to provide the sample complexity of~\ref{generalopt1} for GMRFs. 

\begin{theorem}[Sample Complexity of GMRFs]\label{thm_gmrf_sample}
	Suppose that a sequence of data samples $\left\{x_t^{(i)}\right\}_{i=1}^{N_t}$ are collected from the distribution~\eqref{gmrf} at every $t=0,\dots, T$. Suppose that Assumption~\ref{asp_bound} is satisfied and $\Sigma^\star_t$ is $(s,r)$-weakly sparse for some $s\geq 0$ and $0\leq r<1$.
	 Consider~\ref{generalopt1} with $\widetilde{F}^*(\widehat{\Sigma}_t) = [\texttt{ST}_{\nu_t}(\widehat{\Sigma}_t)]^{-1}$ and parameters $\nu_t \asymp \sqrt{\log n/N_t}$ and $\lambda_t\asymp \mathcal{E}_{\mathrm{g}}\sqrt{\log n/N_t}$. The following statements hold with probability of $1-4Tn^{-15}$:
\begin{itemize}
\item[-] {\bf Estimation error:} Suppose that $N_t\gtrsim \mathcal{N}_{\mathrm{g}}\log p$ for every $t=0,1,\dots,T$. For any $q\in \{0\}\cup [1,\infty)$, the solution of ~\ref{generalopt1} with a temporal $\ell_q$-regularizer satisfies
\begin{align}\label{eq_gmrf_est}
	\norm{\widehat{\Theta}_t-\Theta^\star_t}_{\infty/\infty}\lesssim \mathcal{E}_{\mathrm{g}}\sqrt{\frac{\log n}{N_t}},\qquad \text{ for every } t = 0,\dots,T.
\end{align}
\item[-] {\bf Sparsistency for smoothly-changing GMRFs.} Suppose that the GMRF is smoothly changing with parameters $(q,D), q\geq 1, D\geq 0$. Suppose that $N_t\gtrsim \left(\mathcal{N}_{\mathrm{g}}\vee (\mathcal{E}_{\mathrm{g}}/\bar{\Theta}_t)^2\right)\log n$ for every $t=0,\dots,T$. Then, the solution of ~\ref{generalopt1} with $0<\gamma<1/(1+D)$ and a temporal $\ell_q$-regularizer satisfies
\begin{align}\label{eq_gmrf_lq}
	\mathrm{supp}(\widehat{\Theta}_t) = \mathrm{supp}({\Theta}^\star_t), \qquad \forall t = 0,\dots,T.
\end{align}
\item[-] {\bf Sparsistency for sparsely-changing GMRFs.} Suppose that the GMRF is sparsely changing with parameter $D_0\geq 0$. Suppose that $N_t\gtrsim \left(\mathcal{N}_{\mathrm{g}}\vee (\mathcal{E}_{\texttt{g}}/{\Theta}^{\min}_t)^2\right)\log n$ for every $t=0,\dots,T$. Then, with any choice of $0<\gamma<1$, the optimal solution of~\ref{generalopt1} with temporal $\ell_0$-regularizer satisfies:
				\begin{equation}\label{eq_gmrf_l0}
					\begin{aligned}
					& \mathrm{supp}(\widehat{\Theta}_t) = \mathrm{supp}({\Theta}^\star_t), && \qquad \forall t = 0,\dots,T,\\
					& \mathrm{supp}(\widehat{\Theta}_t-\widehat{\Theta}_{t-1}) = \mathrm{supp}({\Theta}^\star_t-{\Theta}^\star_{t-1}), && \qquad \forall t = 1,\dots,T.
					\end{aligned}
				\end{equation}
\end{itemize}	
\end{theorem}
\begin{proof}
The proof is a direct consequence of Lemma~\ref{lem_bound_gmrf} and Theorem~\ref{thm_constrained}. The details are omitted for brevity.
\end{proof}

\begin{remark}
    We note that Theorem~\ref{thm_gmrf_sample} \textit{does not} impose any condition on the weak sparsity parameter $s$ of the true covariance matrices; instead, it uses this parameter to control the sample complexity of our method. In other words, our theoretical result still holds for large values of $s$, provided that the number of samples scales accordingly.
\end{remark}

	\subsection{Approximate Backward Mapping for DMRFs}\label{subsec_dmrf}
	Suppose that at any given time $t$, a sequence of discrete data samples $\left\{x_t^{(i)}\right\}_{i=1}^{N_t}$ are collected from a DMRF with distribution~\eqref{dmrf}. Our next lemma shows that the empirical mean parameters---corresponding to the empirical node- and edge-wise marginal probabilities---satisfy Assumption~\ref{asp_exp}. 
	\begin{lemma}[Hoeffding's inequality]
	For every $1\leq i,j\leq p$ and $k,l\in\mathcal{K}$, define $\widehat{\mu}_{t;ik} = (1/N_t)\sum_{s=1}^{N_t}\mathbb{I}[x^{(s)}_{t;i} = k]$ and $\widehat{\mu}_{t;ijkl} = (1/N_t)\sum_{s=1}^{N_t}\mathbb{I}[x^{(s)}_{t;i} = k, x^{(s)}_{t;j} = l]$ as the empirical node- and edge-wise marginal probabilities. Then, for all $\widehat{\mu}_{t;ik}$ and $\widehat{\mu}_{t;ijkl}$, Assumption~\ref{asp_exp} holds with $c'=2$, $c = 2$, and $v = \infty$.
	\end{lemma}
	\begin{proof}
	    Note that both $\widehat{\mu}_{t;ik}$ and $\widehat{\mu}_{t;ijkl}$ are empirical averages of binary random variables. Therefore, the proof follows from a direct application of Hoeffding's inequality; see e.g.~\cite[Proposition 2.5]{wainwright2019high}.
	\end{proof}
	\citet{bresler2014hardness} showed that, unlike GMRFs, obtaining the backward mapping of DMRFs is NP-hard, even if random variables are restricted to binary values. Therefore, one can only hope to provide an approximation of this backward mapping. A promising candidate for such approximation is the so-called \textit{tree-reweighted entropy} mapping
$\widetilde{F}^*$
	introduced by~\citet{wainwright2003tree}:
	\begin{align}\label{eq_tre}
		\left[\widetilde{F}^*_{\mathrm{trw}}(\widehat{\mu}_{t})\right]_{\phi} = \begin{cases}
			 \log \widehat{\mu}_{t;ik},
			 & \text{if}\quad \phi = ik\\
			 \rho_{t;ij}\log\left(\frac{\widehat{\mu}_{t;ijkl}}{\widehat{\mu}_{t;ik}\widehat{\mu}_{t;jl}}\right), 
			 & \text{if}\quad \phi = ijkl
		\end{cases}
	\end{align}
The weights $0\leq \rho_{t;ij}\leq 1$ are selected from the so-called \textit{spanning tree polytope}. 
For any spanning tree $\mathcal{T}_t$ of the Markov graph $\mathcal{G}_t$, one can set $\rho_{t;ij} = 1$ if $(i,j)\in \mathcal{T}_t$ and $\rho_{t;ij} = 0$ otherwise. However, due to the unknown nature of $\mathcal{G}_t$, its spanning trees are also unknown. To circumvent this issue, it is common to select $\rho_{t;lk} = 1$ for all tuples $(t,l,k)$~\citep{wainwright2002stochastic}. 

Due to the intractability of the backward mapping, in general one cannot expect the defined approximate backward mapping to be tight even if $\widehat{\mu}_t = \mu^\star_t$. This implies that, unlike GMRFs, it may not be possible to push $\Delta_t$ towards zero. Nonetheless, we will establish the local Lipschitzness of the above approximate backward mapping under the following mild condition.

\begin{assumption}\label{asp_lb}
	There exists $\mu_{\min}>0$ such that:
	\begin{itemize}
		\item We have $\mu^\star_{t;ik} = P(X_i=k)\geq \mu_{\min}$ for every $i,k$.
		\item We have $\mu^\star_{t;ijkl} = P(X_i=k, X_j=l)\geq \mu_{\min}$ for every $i,j,k,l$.
	\end{itemize}
\end{assumption}
The above assumption requires the density of the joint probability distribution to be strictly positive on every atom of the distribution. Given the above assumption, we next prove the local Lipschitzness of our defined approximate backward mapping.
\begin{lemma}\label{lem_dmrf_lip}
	Suppose that Assumption~\ref{asp_lb} is satisfied. Then, the approximate backward mapping~\eqref{eq_tre} is $(L,\alpha)$-locally Lipschitz with $L = 6/\mu_{\min}$ and $\alpha = \mu_{\min}/2$.
\end{lemma}

\begin{proof}
	Suppose that $\phi = ijkl$ is an arbitrary index. For any $\mu, \mu'\in\mathcal{B}_{\mu^\star_t}(\alpha)$, we have 
	\begin{align*}
		\left|\left[\widetilde{F}^*_{\mathrm{trw}}({\mu})\right]_{\phi}-\left[\widetilde{F}^*_{\mathrm{trw}}({\mu'})\right]_{\phi}\right| &= \left|\rho_{t;ij}\log\left(\frac{{\mu}_{ijkl}}{{\mu}_{ik}{\mu}_{jl}}\right)-\rho_{t;ij}\log\left(\frac{{\mu}'_{ijkl}}{{\mu}'_{ik}{\mu}'_{jl}}\right)\right|\nonumber\\
		&\leq \left|\log\left(\frac{\mu'_{ijkl}}{\mu_{ijkl}}\right)\right|+\left|\log\left(\frac{\mu'_{ik}}{\mu_{ik}}\right)\right|+\left|\log\left(\frac{\mu'_{jl}}{\mu_{jl}}\right)\right|
	\end{align*}
On the other hand, one can write
\begin{align*}
	\left|\log\left(\frac{\mu'_{ijkl}}{\mu_{ijkl}}\right)\right|&\leq \left|\log\left(\frac{\min\{\mu'_{ijkl},\mu_{ijkl}\}+|\mu'_{ijkl}-\mu_{ijkl}|}{\min\{\mu'_{ijkl},\mu_{ijkl}\}}\right)\right|\\
	&\leq \left|\log\left(1+\frac{|\mu'_{ijkl}-\mu_{ijkl}|}{\min\{\mu'_{ijkl},\mu_{ijkl}\}}\right)\right|\\
	&\leq \frac{|\mu'_{ijkl}-\mu_{ijkl}|}{\min\{\mu'_{ijkl},\mu_{ijkl}\}}\\
	&\leq \left({2}/{\mu_{\min}}\right)|\mu'_{ijkl}-\mu_{ijkl}|
\end{align*}
where the third inequality follows from $\log(1+x)\leq x$ for any $x>-1$ and the last inequality is due to $\mu, \mu'\in\mathcal{B}_{\mu^\star_t}(\alpha)$ and our choice of $\alpha$.  Similarly, one can show that $\left|\log\left({\mu'_{ik}}/{\mu_{ik}}\right)\right| \leq \left({2}/{\mu_{\min}}\right)|\mu'_{ik}-\mu_{ik}|$ and $\left|\log\left({\mu'_{jl}}/{\mu_{jl}}\right)\right| \leq \left({2}/{\mu_{\min}}\right)|\mu'_{jl}-\mu_{jl}|$. This implies that 
$$
\left|\left[\widetilde{F}^*_{\mathrm{trw}}({\mu})\right]_{ijkl}-\left[\widetilde{F}^*_{\mathrm{trw}}({\mu'})\right]_{ijkl}\right|\leq (6/\mu_{\min})\cdot\max\{|\mu'_{ijkl}-\mu_{ijkl}|, |\mu'_{ik}-\mu_{ik}|, |\mu'_{jl}-\mu_{jl}|\}
$$
Similarly, for an arbitrary index $\phi = ik$, we have $\left|\left[\widetilde{F}^*_{\mathrm{trw}}({\mu})\right]_{\phi}-\left[\widetilde{F}^*_{\mathrm{trw}}({\mu'})\right]_{\phi}\right|\leq (2/\mu_{\min})|\mu'_{ik}-\mu_{ik}|$.
This implies that
$$
\norm{\widetilde{F}^*_{\mathrm{trw}}({\mu})-\widetilde{F}^*_{\mathrm{trw}}({\mu'})}_\infty\leq (6/\mu_{\min})\norm{\mu-\mu'}_\infty,
$$
thereby completing the proof.
\end{proof}
Lemma~\ref{lem_dmrf_lip} combined with Proposition~\ref{lem_back_dev} leads to the following concentration bound on the backward mapping deviation.
\begin{sloppypar}
	\begin{lemma}\label{lem_dmrf_backward_dev}
		Suppose that Assumption~\ref{asp_lb} is satisfied. Moreover, suppose that $N_t\geq (16/\mu_{\min}^2)\log p$. Then, we have
		\begin{align}\label{eq_backward_dmrf}
			\norm{\theta_t^\star-\widetilde{F}^*_{\mathrm{trw}}(\widehat{\mu}_t)}_\infty\leq \Delta_t+\frac{12}{\mu_{\min}}\sqrt{\frac{\log p}{N_t}},\qquad \text{ with probability of $1-2p^{-8}$.}
		\end{align}
	\end{lemma}
\end{sloppypar}
\begin{proof}
	The proof immediately follows from Lemma~\ref{lem_dmrf_lip} and Proposition~\ref{lem_back_dev}.
\end{proof}
Lemma~\ref{lem_dmrf_backward_dev} implies that the backward mapping deviation can be upper bounded with $2\Delta_t$, so long as $N_t\geq \left(12/(\mu_{\min}\Delta_t)\right)^2\log p$. The above result implies that, with logarithmic sample size, the backward mapping deviation stays in the order of the approximation error $\Delta_t$. 
Equipped with the above lemma, we are ready to present the sample complexity of \ref{generalopt1} for DMRFs. Similar to GMRFs, we define the following quantities:
\begin{align*}
    \bar\theta_t := \min_{(i,j)\in\mathcal{S}_t}|\theta^\star_{t;ij}|,\quad &\Delta\bar\theta_t := \min_{(i,j)\in\mathcal{D}_t}|\theta^\star_{t;ij}-\theta^\star_{t-1;ij}|,\quad  \theta_t^{\min} := \begin{cases}
    \min\left\{\bar\theta_t, \frac{\Delta\bar\theta_t}{2}, \frac{\Delta\bar\theta_{t-1}}{2}\right\} & t\geq 1\\
    \min\left\{\bar\theta_t, \frac{\Delta\bar\theta_t}{2}\right\} & t=0
    \end{cases}
\end{align*}

\begin{theorem}[Sample Complexity of DMRFs]\label{thm_dmrf_sample}
	Suppose that a sequence of data samples $\left\{x_t^{(i)}\right\}_{i=1}^{N_t}$ are collected from the distribution~\eqref{dmrf} at every $t=0,\dots, T$. Suppose that Assumption~\ref{asp_lb} is satisfied. Consider~\ref{generalopt1} with the backward mapping $\widetilde{F}_{\mathrm{trw}}^*$ defined as~\eqref{eq_tre} and parameter $\lambda_t\asymp \Delta_t+\frac{1}{\mu_{\min}}\sqrt{\frac{\log p}{N_t}}$. The following statements hold with probability of $1-2Tp^{-8}$:
\begin{itemize}
\item[-] {\bf Estimation error:} Suppose that $N_t\gtrsim \mu_{\min}^{-2}\log p$ for every $t = 0,\dots, T$. For any $q\in \{0\}\cup [1,\infty)$, the solution of ~\ref{generalopt1} with a temporal $\ell_q$-regularizer satisfies
\begin{align*}
			\norm{\widehat{\theta}_t-\theta^\star_t}_{\infty}\lesssim \Delta_t+\frac{1}{\mu_{\min}}\sqrt{\frac{\log p}{N_t}},\qquad \text{ for every } t = 0,\dots,T.
\end{align*}
\item[-] {\bf Sparsistency for smoothly-changing DMRFs.} Suppose that the time-varying DMRF is smoothly changing with parameters $(q,D), q\geq 1, D\geq 0$. Suppose that $\Delta_t\leq \bar\theta_t/2$ and $N_t\gtrsim (\mu_{\min}\cdot\bar\theta_t)^{-2}\log p$ for every $t = 0,\dots, T$. Then, the solution of ~\ref{generalopt1} with $0<\gamma<1/(1+D)$ and a temporal $\ell_q$-regularizer satisfies
\begin{align}\label{eq_dmrf_lq}
	\mathrm{supp}\left(\widehat{\theta}_t\right) = \mathrm{supp}\left({\theta}^\star_t\right), \qquad \forall t = 0,\dots,T.
\end{align}
\item[-] {\bf Sparsistency for sparsely-changing DMRFs.} Suppose that the time-varying DMRF is sparsely changing with parameter $D_0\geq 0$. Suppose that $\Delta_t\leq \theta^{\min}_t/2$ and $N_t\gtrsim (\mu_{\min}\cdot\theta^{\min}_t)^{-2}\log n$ for every $t = 0,1,\dots, T$. Then, with any choice of $0<\gamma<1$, the optimal solution of~\ref{generalopt1} with temporal $\ell_0$-regularizer satisfies:
				\begin{equation}\label{eq_dmrf_l0}
					\begin{aligned}
					& \mathrm{supp}\left(\widehat{\theta}_t\right) = \mathrm{supp}({\theta}^\star_t), && \qquad \forall t = 0,\dots,T,\\
					& \mathrm{supp}\left(\widehat{\theta}_t-\widehat{\theta}_{t-1}\right) = \mathrm{supp}({\theta}^\star_t-{\theta}^\star_{t-1}), && \qquad \forall t = 1,\dots,T.
					\end{aligned}
				\end{equation}
\end{itemize}
\end{theorem}

\begin{proof}
	The proof is a direct consequence of Lemma~\ref{lem_dmrf_backward_dev} and Theorem~\ref{thm_constrained}. The details are omitted for brevity.
\end{proof}
According to Theorem~\ref{thm_dmrf_sample}, the exact sparsity recovery is only possible if the approximation error of the backward mapping $\Delta_t$ is small. 

\subsection{Kernel Averaging}\label{sec:kernelAvg}
In most realistic time-varying MRFs, the underlying graphical model changes continuously over time as the samples continue to arrive. For instance, in financial markets, the underlying stock correlation network may change quickly in response to global events. Therefore, a stock holder needs to identify the sharp changes in the market and rebalance their portfolio ``on the go''~\citep{talih2005structural, hallac2017network}. 
Evidently, in these applications, the sample size $N_t$ may be significantly smaller than what is required for \ref{generalopt1} to provide a reliable estimation of the canonical parameters.

To address the scarcity of data in this setting, we leverage and combine the information provided by the samples over time. In particular, we propose to replace the empirical mean parameters with their weighted averages over time, where the weights are obtained from a nonparametric kernel $K:\mathbb{R}\to\mathbb{R}_+$. Without loss of generality,  suppose that $N_t = 1$\footnote{If $N_t > 1$, the sufficient statistics $\phi(x_s)$ in~\eqref{eq_ker} can be replaced by $\frac{1}{N_t}\sum_{i=1}^{N_t}\phi(x^{(i)}_s)$.} for every $t=0,\dots,T$ and consider the weighted mean parameter
\begin{align}\label{eq_ker}
	\widehat{\mu}^{\mathrm{ker}}_t = \sum_{s=0}^Tw(s,t)\phi(x_s),\quad \text{where }\  w(s,t) = \frac{1}{Th}K\left(\frac{s-t}{Th}\right).
\end{align}
Here, $K(\cdot)$ is a symmetric nonnegative kernel that satisfies a set of mild conditions which will be delineated later. Moreover, the parameter $h$ is the \textit{bandwidth} of the kernel, controlling the decay rate of the weights.  The key insight behind kernel averaging is simple: at any given time $t$, we estimate the mean parameters by taking the weighted average of the samples over time, where the weights are obtained from a kernel that assigns smaller weights to samples that are temporally farther away from $t$. The idea of kernel averaging (also known as \textit{kernel smoothing}) dates back to late 60s~\citep{epanechnikov1969non, altman1992introduction, wand1994kernel}, and has been recently used in graphical model inference~\citep{greenewald2017time, zhou2010time}.

To ensure that such weighted averaging is statistically consistent, we must assume that its true counterpart changes slowly over time to ensure that the samples collected over time can be used to reveal partial information about the true mean parameter at time $t$. More precisely, our hope is to guarantee that $\sum_{s=0}^Tw(s,t)\phi(x_s)\approx \mathbb{E}[\phi(x_t)] = \mu^\star_t$. To formalize this intuition, we assume that $\mu^\star_t = \mu(t/T)$, for some twice continuously differentiable $\mu(x):[0,1]\to \mathbb{R}^n$ that satisfies the following assumption:\footnote{A function can be differentiable only on open sets. Without loss of generality and with a slight abuse of notation, we replace the derivatives of $\mu(x)$ at $x=0$ and $x=1$ with their right and left derivatives, respectively.}
\begin{assumption}[Bounded derivatives]\label{asp_bounded_der}
	There exist constants $0\leq \Gamma_0, \Gamma_1,\Gamma_2<\infty$ such that $\norm{\mu(x)}_\infty\leq \Gamma_0$, $\norm{\mu'(x)}_\infty\leq \Gamma_1$, and $\norm{\mu''(x)}_\infty\leq \Gamma_2$, for every $0\leq x\leq 1$. 
\end{assumption}
The above assumption implies that the function $\mu$ and its first derivative change smoothly. 
\begin{assumption}[Kernel properties]\label{asp_kernel}
	The kernel $K(x)$ satisfies the following conditions:
	\begin{itemize}
		\item[-] $\int_{-1}^{1} K(x)dx=1$;
		\item[-] There exist constants $K_l, K_u,K_1,K_2<\infty$ such that:
		\begin{align*}
		    0<K_l\leq K(x/h) \leq K_u, \quad |K'(x/h)| \leq K_1\cdot h^{-1}, \quad |K''(x/h)| \leq K_2\cdot h^{-2},
		\end{align*}
		for every $-h\leq x\leq h$ and $0<h\leq 1$. 
	\end{itemize}
\end{assumption} 
The above assumption holds for most standard kernels. For instance, it is easy to verify that the uniform kernel $K(x) = 1/2$ with domain $[-1,1]$ satisfies Assumption~\ref{asp_kernel} with parameters $K_l = K_u = 1/2$ and $K_1 = K_2 = 0$. As another example, consider the truncated Gaussian kernel $K(x) = C e^{-x^2/2}$, where $C = (\Phi(1)-\Phi(-1))^{-1}\approx 1.465$ is the normalizer of the kernel and $\Phi$ is the CDF of the normal distribution. Simple calculation reveals that this kernel satisfies Assumption~\ref{asp_kernel} with parameters $K_l = Ce^{-1/2}$, $K_u = K_1 = C$ and $K_2 = 2C$. 

Given the above two assumptions, our next goal is to derive a generalization of Proposition~\ref{lem_back_dev} to the kernel averaged empirical mean parameters. Recall that Proposition~\ref{lem_back_dev} relies on an upper bound on $\norm{\widehat{\mu}_t-{\mu}^\star_t}_\infty$. With kernel averaging, the empirical mean parameter $\widehat{\mu}_t$ is replaced with $\widehat{\mu}_t^{\mathrm{ker}}$. As a result, we need to obtain a concentration bound on the deviation of $\widehat{\mu}_t^{\mathrm{ker}}$ from $\mu^\star_t$. To this goal, we write
\begin{align*}
    \norm{\widehat{\mu}_t^{\mathrm{ker}}-\mu^\star_t }_\infty\leq \underbrace{\left\|\mu^\star_t - \mathbb{E}[\widehat{\mu}_t^{\mathrm{ker}}]\right\|_{\infty}}_{\text{kernel error}}+\underbrace{\left\|\widehat{\mu}_t^{\mathrm{ker}} - \mathbb{E}[\widehat{\mu}_t^{\mathrm{ker}}]\right\|_{\infty}}_{\text{deviation from kernel mean}}
\end{align*}
Our next two lemmas provide upper bounds on the kernel error and the deviation of kernel mean.
\begin{lemma}[Kernel error]\label{lem_ker_exp}
	Under Assumptions~\ref{asp_bounded_der} and~\ref{asp_kernel}, we have
	\begin{align}
		\left\|\mu^\star_t - \mathbb{E}[\widehat{\mu}_t^{\mathrm{ker}}]\right\|_{\infty}\leq \frac{\Gamma_2}{4}\cdot h^2+{\Gamma_1}\cdot h+\frac{\Gamma_2K_u}{24 T^2}\cdot h^{-1}+\frac{\Gamma_1K_1}{12 T^2}\cdot h^{-3}+\frac{\Gamma_0K_2}{24 T^2}\cdot h^{-5}
	\end{align}
\end{lemma}

\begin{proof}
	One can write
	\begin{align*}
		\left\|\mu^\star_t - \mathbb{E}[\widehat{\mu}_t^{\mathrm{ker}}]\right\|_{\infty}\leq& \underbrace{\left\|\mu^\star_t - \int_{0}^{1}\frac{1}{h}K\left(\frac{x-t/T}{h}\right)\mu(x)dx\right\|_{\infty}}_{\mathcal{A}}\\
		&+\underbrace{\left\|\int_{0}^{1}\frac{1}{h}K\left(\frac{x-t/T}{h}\right)\mu(x)dx-\mathbb{E}[\widehat{\mu}_t^{\mathrm{ker}}]\right\|_{\infty}}_{\mathcal{B}}
	\end{align*}
According to \citet[Lemma 14]{zhou2010time} we have $\mathcal{A}\leq \frac{\Gamma_2}{4}\cdot h^2+{\Gamma_1}\cdot h$. Therefore, it remains to control $\mathcal{B}$. To this goal, define $f(x):=\frac{1}{h}K\left(\frac{x-t/T}{h}\right)\mu(x)$. Note that 
\begin{align*}
	\mathbb{E}[\widehat{\mu}_t^{\mathrm{ker}}] = \frac{1}{T}\sum_{s=0}^T\frac{1}{h}K\left(\frac{s-t}{Th}\right)\mu\left(\frac{s}{T}\right) = \frac{1}{T}\sum_{s=0}^Tf\left(\frac{s}{T}\right).
\end{align*}
This implies that $\mathbb{E}[\widehat{\mu}_t^{\mathrm{ker}}]$ is the Riemann approximation of $\int_{0}^1 f(x)dx$. Therefore, we have 
\begin{align}\label{riemann_bound}
	\mathcal{B} = \left\|\int_{0}^{1}f(x)dx - \frac{1}{T}\sum_{s=0}^Tf\left(\frac{s}{T}\right)\right\|_\infty\leq \frac{\max_{0\leq x\leq 1}\|f''(x)\|_\infty}{24T^2}
\end{align}
where the second inequality follows from a standard bound  on the Riemann sum~\citep{hughes2020calculus}. Next, we provide an upper bound on $\max_{0\leq x\leq 1}\|f''(x)\|_\infty$. It is easy to verify that
\begin{align}\label{eq_f_secondder}
	f''(x) = h^{-3}\cdot K''\left(\frac{x-t/T}{h}\right)\mu(x)+2h^{-2}\cdot K'\left(\frac{x-t/T}{h}\right)\mu'(x)+h^{-1}K\left(\frac{x-t/T}{h}\right)\mu''(x),
\end{align}
Therefore, with Assumption~\ref{asp_kernel}, we have
\begin{align*}
	\|f''(x)\|_\infty\leq h^{-5}\cdot K_2\Gamma_0+2h^{-3}\cdot K_1\Gamma_1+h^{-1}\cdot K_u\Gamma_2.
\end{align*}
The above inequality combined~\eqref{riemann_bound} completes the proof.
\end{proof}

\begin{lemma}[Deviation from kernel mean]\label{lem_ker_dev}
Suppose that Assumptions~\ref{asp_bounded_der} and~\ref{asp_kernel} are satisfied. Moreover, suppose that $T\geq (\tau+1)(K_u/K_l)^4\log p$ for an arbitrary $\tau>0$. Then, we have
\begin{align*}
\norm{\widehat{\mu}_t^{\mathrm{ker}}-\mathbb{E}[\widehat{\mu}_t^{\mathrm{ker}}]}_\infty\leq \sqrt{\frac{(\tau+1)K_u^2\log p}{cTh^2}},\quad \text{with probability}\ 1-c'p^{-\tau}
\end{align*}
\end{lemma}
\begin{proof}
Let us define $g_s = (1/h)K((s-t)/(Th))$ and $Y_s = g_s\phi(x_s)$. We have $\widehat{\mu}_t^{\mathrm{ker}} = (1/T)\sum_{s=0}^T Y_s$. According to Assumption~\ref{asp_exp}, $Y_s$ satisfies
\begin{align*}
\mathbb{P}(|Y_{s;i} - \mathbb{E}[Y_{s;i}]|\geq \delta)\leq c'\exp(-(c/g_s^2)\delta^2)\quad \text{for every}\quad \delta\in(0,g_sv],\ \ i=1,\dots,n.
\end{align*}
Therefore, $Y_{s;i}$ is a sub-exponential random variable with parameters $(g_s^2/(2c), 1/(g_sv))$; see~\cite[Chapter 2]{wainwright2019high} for the definition of sub-exponential random variables). On the other hand, it is easy to see that $\mathbb{E}[\widehat{\mu}^{\mathrm{ker}}_{t;i}] = \frac{1}{T}\sum_{s=0}^Tg_s\mu_s =  \frac{1}{T}\sum_{s=0}^T\mathbb{E}[Y_{s;i}]$. Therefore, one can invoke the concentration bound on the sum of sub-exponential random variables~\cite[Proposition 2.9]{wainwright2019high} to obtain
\begin{align*}
	\mathbb{P}\left(\left|\frac{1}{T}\sum_{s=0}^TY_{s;i} - \mathbb{E}[\widehat{\mu}^{\mathrm{ker}}_{t;i}]\right|\geq \delta\right)\leq c'\exp\left(-\frac{cT^2}{\sum_{s=0}^T g_s^2}\cdot\delta^2\right)\quad \text{for every}\quad \delta\in\left(0,\frac{\sum_{s=0}^T g_s^2}{T\max_{s}g_s}\right].
\end{align*}
A simple union bound leads to 
\begin{multline*}
	\mathbb{P}\left(\left\|\frac{1}{T}\sum_{s=0}^TY_{s} - \mathbb{E}[\widehat{\mu}^{\mathrm{ker}}_{t}]\right\|_\infty\geq \delta\right)\leq c'\exp\left(\log p-\frac{cT^2}{\sum_{s=0}^T g_s^2}\cdot\delta^2\right)\\ \text{for every}\quad \delta\in\left(0,\frac{\sum_{s=0}^T g_s^2}{T\max_{s}g_s}\right].
\end{multline*}
Upon choosing $\delta = \sqrt{\frac{(\tau+1)\left(\sum_{s=0}^T g_s^2\right)\log p}{T^2 c}}$, we have
\begin{align*}
	\left\|\frac{1}{T}\sum_{s=0}^TY_{s} - \mathbb{E}[\widehat{\mu}^{\mathrm{ker}}_{t}]\right\|_\infty\leq \sqrt{\frac{(\tau+1)\left(\sum_{s=0}^T g_s^2\right)\log p}{T^2 c}},
\end{align*}
with probability $1-c'p^{-\tau}$, provided that 
\begin{align}\label{eq_delta}
	\sqrt{\frac{(\tau+1)\left(\sum_{s=0}^T g_s^2\right)\log p}{T^2 c}}\leq \frac{\sum_{s=0}^T g_s^2}{T\max_{s}g_s}.
\end{align}
To verify the last inequality, recall that according to Assumption~\ref{asp_kernel}, we have $K_l\leq K(\cdot)\leq K_u$, which in turn implies that $h^{-1}K_l\leq g_s\leq h^{-1}K_u$ for every $s$. This results in
\begin{align*}
	\sqrt{\frac{(\tau+1)\left(\sum_{s=0}^T g_s^2\right)\log p}{T^2 c}}\leq\sqrt{\frac{(\tau+1)K_u^2\log p}{Th^2 c}}, \qquad  \frac{K_l^2}{hK_u}\leq \frac{\sum_{s=0}^T g_s^2}{T\max_{s}g_s}.
\end{align*}
Therefore, to satisfy~\eqref{eq_delta}, it suffices to have 
\begin{align*}
	\sqrt{\frac{(\tau+1)K_u^2\log p}{Th^2 c}}\leq \frac{K_l^2}{hK_u}
\end{align*}
which is satisfied with $T\geq (\tau+1)(K_u/K_l)^4\log p$. This completes the proof.
\end{proof}
Combining Lemmas~\ref{lem_ker_exp} and~\ref{lem_ker_dev} leads to an upper bound on $\norm{\widehat{\mu}_t^{\mathrm{ker}}-\mu_t^\star}_\infty$. To simplify our subsequent arguments, we define $\mathcal{C}^{\mathrm{ker}} = \Gamma_2/4+\Gamma_1+\Gamma_2K_u/24+\Gamma_1K_1/12+\Gamma_0K_2/24$.

\begin{lemma}\label{lem_backward_ker}
	Suppose that Assumptions~\ref{asp_bounded_der} and~\ref{asp_kernel} are satisfied, and $T\geq (\tau+1)(K_u/K_l)^4\log p$ for any arbitrary $\tau\geq 0$. Then, upon choosing $h = T^{-1/4}$, we have 
	\begin{align}
		\norm{\widehat{\mu}_t^{\mathrm{ker}}-\mu_t^\star}_\infty\leq \frac{\mathcal{C}^{\mathrm{ker}}}{T^{1/4}}+\frac{K_u\sqrt{(\tau+1)\log p}}{\sqrt{c}T^{1/4}},\quad \text{with probability of $1-c'p^{-\tau}$.}
	\end{align}
\end{lemma}
\begin{proof}
	The proof follows from Lemmas~\ref{lem_ker_exp} and~\ref{lem_ker_dev} after replacing $h = T^{-1/4}$.
\end{proof}
Given Lemma~\ref{lem_backward_ker}, we are now ready to provide an upper bound for the backward mapping deviation of the kernel averaged empirical mean parameter.
\begin{proposition}[Backward mapping deviation with kernel averaging]\label{lem_back_dev_ker}
	Suppose that Assumptions~\ref{asp_bounded_der} and~\ref{asp_kernel} are satisfied and $\widetilde{F}^*$ is $(L,\alpha)$-locally Lipschitz for some $\alpha, L>0$. Then, for any $\tau\geq 0$ and $T\geq \max\left\{16(\mathcal{C}^{\mathrm{ker}}/\alpha)^4, \left(\frac{16K_u^4(\tau+1)^2}{c^2K_l^4\alpha^4}\cdot \log^2 p\right)\right\}$, we have
	\begin{align*}
		\norm{\theta_t^\star-\widetilde{F}^*(\widehat{\mu}^{\mathrm{ker}}_t)}_\infty\leq \underbrace{\norm{\theta_t^\star-\widetilde{F}^*({\mu}^\star_t)}_\infty}_{:=\Delta_t}+\frac{\mathcal{C}L}{T^{1/4}}+\frac{K_uL\sqrt{(\tau+1)\log p}}{\sqrt{c}T^{1/4}},
	\end{align*}
with probability of $1-c'p^{-\tau}$.
\end{proposition}
\begin{proof}
	Due to Lemma~\ref{lem_backward_ker} and $T\geq \max\left\{16(\mathcal{C}^{\mathrm{ker}}/\alpha)^4, \left(\frac{16K_u^4(\tau+1)^2}{c^2K_l^4\alpha^4}\cdot \log^2 p\right)\right\}$, we have
	\begin{align*}
		\norm{\widehat{\mu}_t^{\mathrm{ker}}-\mu_t^\star}_\infty\leq \frac{\mathcal{C}}{T^{1/4}}+\frac{K_u\sqrt{(\tau+1)\log p}}{\sqrt{c}T^{1/4}}\leq \alpha
	\end{align*}
with probability of $1-c'p^{-\tau}$. Conditioned on this event, we have
\begin{align*}
	\norm{\theta_t^\star-\widetilde{F}^*(\widehat{\mu}^{\mathrm{ker}}_t)}_\infty &\leq \norm{\theta_t^\star-\widetilde{F}^*({\mu}^\star_t)}_\infty+\norm{\widetilde{F}^*(\widehat{\mu}^{\mathrm{ker}}_t)-\widetilde{F}^*({\mu}^\star_t)}_\infty\\
	&\leq \norm{\theta_t^\star-\widetilde{F}^*({\mu}^\star_t)}_\infty+L\norm{\widehat{\mu}^{\mathrm{ker}}_t-{\mu}^\star_t}_\infty\\
	&\leq \norm{\theta_t^\star-\widetilde{F}^*({\mu}^\star_t)}_\infty+\frac{\mathcal{C}L}{T^{1/4}}+\frac{K_uL\sqrt{(\tau+1)\log p}}{\sqrt{c}T^{1/4}}
\end{align*}
where in the second inequality we used the $(L,\alpha)$-local Lipschitzness of $\widetilde{F}^*$.
\end{proof}

Relying on Proposition~\ref{lem_back_dev_ker}, we next present the sample complexity of \ref{generalopt1} for GMRF and DMRF with kernel averaging. To this goal, we first provide the explicit form of the kernel empirical mean parameter~\eqref{eq_ker} in both settings:
\begin{sloppypar}
	\begin{itemize}
		\item For GMRF, the kernel mean parameters are the weighted sample covariance matrices defined as $\widehat{\Sigma}^{\mathrm{ker}}_t = \sum_{s=0}^T w(s,t)x_sx_s^\top$, for every $t = 0,\dots,T$.
		\item For DMRF, the kernel mean parameters are the weighted empirical node- and edge-wise marginal probabilities respectively defined as $\widehat{\mu}^{\mathrm{ker}}_{t;ik} = \sum_{s=0}^{T}w(s,t)\mathbb{I}[x_{s;i} = k]$ and $\widehat{\mu}^{\mathrm{ker}}_{t;ijkl} = \sum_{s=0}^{T}w(s,t)\mathbb{I}[x_{s;i} = k, x_{s;j} = l]$, for every $i,j,k,l$.
	\end{itemize}
\end{sloppypar}

To streamline the presentation, we define the following quantities:
\begin{align*}
\mathcal{E}^{\mathrm{ker}}_{\mathrm{g}} := \frac{\mathcal{C}^{\mathrm{ker}}+\kappa_1+K_u\kappa_3^{-1}}{\kappa_2^2},\quad \mathcal{N}^{\mathrm{ker}}_{\mathrm{g}} :=  \max\left\{\left(\frac{\mathcal{C}^{\mathrm{ker}}s}{\kappa_2}\right)^{\frac{4}{1-r}}, \left(\frac{K_u s}{K_l\sqrt{\kappa_3}\kappa_2}\right)^{\frac{4}{1-r}}\right\}
\end{align*}
where $\kappa_1,\kappa_2,\kappa_3,r, s$ are defined in Assumptions~\ref{asp_bound} and Definition~\ref{def_weak}.

\begin{proposition}[Sample complexity of GMRF with kernel averaging]\label{prop_gmrf_sample_ker}
	Suppose that a sequence of data samples $\left\{x_t\right\}_{t=0}^{T}$ are collected from the distribution~\eqref{gmrf}. Moreover, suppose that Assumptions~\ref{asp_bound},~\ref{asp_bounded_der}, and~\ref{asp_kernel} are satisfied and $\Sigma^\star_t$ is $(s,r)$-weakly sparse for some $s\geq 0$ and $0\leq r<1$.
	Consider~\ref{generalopt1} with $\widetilde{F}^*(\widehat{\Sigma}^{\mathrm{ker}}_t) = [\texttt{ST}_{\nu_t}(\widehat{\Sigma}^{\mathrm{ker}}_t)]^{-1}$ and parameters $\nu_t \asymp {\sqrt{\log n}}/{T^{1/4}}$, $\lambda_t\asymp {\mathcal{E}_g\log n}/{T^{1/4}}$, and $h \asymp T^{-1/4}$. Then, with probability of $1-4Tn^{-15}$, the following statements hold:
	\begin{itemize}
		\item[-] {\bf Estimation error:} Suppose that $T\gtrsim \mathcal{N}^{\mathrm{ker}}_{\mathrm{g}}\log^2 n$. For any $q\in \{0\}\cup [1,\infty)$, the solution of ~\ref{generalopt1} with a temporal $\ell_q$-regularizer satisfies
		\begin{align*}
			\norm{\widehat{\Theta}_t-\Theta^\star_t}_{\infty/\infty}\lesssim \mathcal{E}^{\mathrm{ker}}_{\mathrm{g}}{\frac{\sqrt{\log n}}{T^{1/4}}},\qquad \text{ for every }\ t = 0,1,\dots,T.
		\end{align*}
		\item[-] {\bf Sparsistency for smoothly-changing GMRF.} Suppose that the GMRF is smoothly changing with parameters $(q,D), q\geq 1, D\geq 0$. Suppose that $T\gtrsim \left(\mathcal{N}^{\mathrm{ker}}_{\mathrm{g}}\vee (\mathcal{E}^{\mathrm{ker}}_{\mathrm{g}}/\bar{\Theta}_t)^4\right)\log^2 n$. Then, the solution of ~\ref{generalopt1} with $0<\gamma<1/(1+D)$ and a temporal $\ell_q$-regularizer satisfies
		\begin{align}
	\mathrm{supp}(\widehat{\Theta}_t) = \mathrm{supp}({\Theta}^\star_t), \qquad \forall t = 0,\dots,T.
\end{align}
\item[-] {\bf Sparsistency for sparsely-changing GMRF.} Suppose that the GMRF is sparsely changing with with sparsity parameter $D_0\geq 0$. Suppose that $T\gtrsim \left(\mathcal{N}^{\mathrm{ker}}_{\mathrm{g}}\vee (\mathcal{E}^{\mathrm{ker}}_{\mathrm{g}}/{\Theta}^{\min}_t)^4\right)\log^2 n$. Then, with any choice of $0<\gamma<1$, the optimal solution of~\ref{generalopt1} with temporal $\ell_0$-regularizer satisfies
\begin{equation}
					\begin{aligned}
					& \mathrm{supp}(\widehat{\Theta}_t) = \mathrm{supp}({\Theta}^\star_t), && \qquad \forall t = 0,\dots,T,\\
					& \mathrm{supp}(\widehat{\Theta}_t-\widehat{\Theta}_{t-1}) = \mathrm{supp}({\Theta}^\star_t-{\Theta}^\star_{t-1}), && \qquad \forall t = 1,\dots,T.
					\end{aligned}
				\end{equation}
	\end{itemize}	
\end{proposition}
\begin{sloppypar}
\begin{proof}
	The proof follows that of Theorem~\ref{thm_gmrf_sample}. First, due to Lemma~\ref{lem_error_lip}, we have $\Delta_t\leq (2\kappa_1/\kappa_2)\nu_t$ and $\widetilde{F}^*(\Sigma)$ is $(L,\alpha)$-locally Lipschitz with $L=4/\kappa_2^2$ and $\alpha = \min\{\kappa_2\nu_t^r/(24s),\nu_t/2\}$, provided that $\nu_t\leq (\kappa_2/(40s))^{1/(1-r)}$. This combined with Proposition~\ref{lem_back_dev_ker} and the selected values for $\nu_t$ and $T$ leads to
	\[
	\norm{\Theta_t^\star-[\texttt{ST}_{\nu_t}(\widehat{\Sigma}^{\mathrm{ker}}_t)]^{-1}}_{\infty/\infty}\leq \mathcal{E}^{\mathrm{ker}}_{\mathrm{g}}\frac{\sqrt{\log n}}{T^{1/4}},\quad \text{for every $t$ with probability of $1-4Tn^{-15}$.}
	\]
Therefore, upon choosing $\lambda_t\asymp {\mathcal{E}^{\mathrm{ker}}_{\mathrm{g}}\log n}/{T^{1/4}}$, Theorem~\ref{thm_constrained} can be invoked to write $\norm{\widehat{\Theta}_t-\Theta^\star_t}_{\infty/\infty}\leq \lambda_t$, which completes the proof of the first statement. 
To prove the correct sparsity recovery, according to Theorem~\ref{thm_constrained} we additionally need to satisfy $\lambda_t\leq \bar{\Theta}_t/2$ for smoothly-changing GMRFs, and $\lambda_t\leq {\Theta}^{\min}_t/2$ for sparsely-changing GMRFs. These are guaranteed to hold with the above error bound and the assumed lower bounds on $T$.
\end{proof}
\end{sloppypar}

Finally, we provide the sample complexity if inferring DMRFs with kernel averaging. Similar to GMRF, we define the following quantities:
\begin{align*}
\mathcal{E}^{\mathrm{ker}}_{\mathrm{d}} := \frac{\mathcal{C}^{\mathrm{ker}}+K_u}{\mu_{\min}},\quad \mathcal{N}^{\mathrm{ker}}_{\mathrm{d}} :=  \max\left\{\left(\frac{\mathcal{C}^{\mathrm{ker}}}{\mu_{\min}}\right)^4, \left(\frac{K_u}{K_l\mu_{\min}}\right)^4\right\}
\end{align*}

\begin{proposition}[Sample complexity of DMRF with kernel averaging]\label{prop_dmrf_sample_ker}
	Suppose that a sequence of data samples $\left\{x_t\right\}_{t=0}^{T}$ are collected from the distribution~\eqref{dmrf}. Moreover, suppose that Assumptions~\ref{asp_lb},~\ref{asp_bounded_der}, and~\ref{asp_kernel} are satisfied.
	Consider~\ref{generalopt1} with parameter $\lambda_t\asymp \Delta_t+\mathcal{E}^{\mathrm{ker}}_{\mathrm{d}}\sqrt{\log p} /T^{1/4}$ and the backward mapping $\widetilde{F}_{\mathrm{trw}}^*$ defined as~\eqref{eq_tre}. Then, with probability of $1-2Tp^{-8}$, the following statements hold:
	\begin{itemize}
		\item[-] {\bf Estimation error:} Suppose that $T\gtrsim \mathcal{N}^{\mathrm{ker}}_{\mathrm{d}}\log^2 p$. For any $q\in \{0\}\cup [1,\infty)$, the solution of ~\ref{generalopt1} with a temporal $\ell_q$-regularizer satisfies
		\begin{align*}
			\norm{\widehat{\theta}_t-\theta^\star_t}_{\infty}\lesssim \mathcal{E}^{\mathrm{ker}}_{\mathrm{d}}{\frac{\sqrt{\log p}}{T^{1/4}}},\qquad \text{ for every }\ t = 0,\dots,T.
		\end{align*}
		\item[-] {\bf Sparsistency for smoothly-changing DMRF.} Suppose that the DMRF is smoothly changing with parameters $(q,D), q\geq 1, D\geq 0$. Suppose that $\Delta_t\leq \bar\theta_t/2$ and $T\gtrsim \left(\mathcal{N}^{\mathrm{ker}}_{\mathrm{d}}\vee (\mathcal{E}^{\mathrm{ker}}_{\mathrm{d}}/\bar{\theta}_t)^4\right)\log^2 p$. Then, the solution of ~\ref{generalopt1} with $0<\gamma<1/(1+D)$ and a temporal $\ell_q$-regularizer satisfies
		\begin{align}\label{eq_dmrf_lq}
	\mathrm{supp}\left(\widehat{\theta}_t\right) = \mathrm{supp}\left({\theta}^\star_t\right), \qquad \forall t = 0,\dots,T.
\end{align}
		\item[-] {\bf Sparsistency for sparsely-changing DMRF.} Suppose that the time-varying DMRF is sparsely changing with sparsity parameter $D_0\geq 0$. Suppose that $\Delta_t\leq \theta_t^{\min}/2$ and $T\gtrsim \left(\mathcal{N}^{\mathrm{ker}}_{\mathrm{d}}\vee (\mathcal{E}^{\mathrm{ker}}_{\mathrm{d}}/{\theta}^{\min}_t)^4\right)\log^2 p$. Then, with any choice of $0<\gamma<1$, the optimal solution of~\ref{generalopt1} with temporal $\ell_0$-regularizer satisfies
		\begin{equation}\label{eq_dmrf_l0}
					\begin{aligned}
					& \mathrm{supp}\left(\widehat{\theta}_t\right) = \mathrm{supp}({\theta}^\star_t), && \qquad \forall t = 0,\dots,T,\\
					& \mathrm{supp}\left(\widehat{\theta}_t-\widehat{\theta}_{t-1}\right) = \mathrm{supp}({\theta}^\star_t-{\theta}^\star_{t-1}), && \qquad \forall t = 1,\dots,T.
					\end{aligned}
				\end{equation}
	\end{itemize}	
\end{proposition}
\begin{proof}
The proof is similar to those of Theorem~\ref{thm_gmrf_sample} and Proposition~\ref{prop_gmrf_sample_ker}. The details are omitted for brevity.
\end{proof}

    The above propositions show that the previously imposed lower bounds on the number of samples per time can be relaxed via kernel averaging. However, this comes at the expense of increasing the error bound from $\cO(N_t^{-1/2})$ to $\cO(T^{-1/4})$ (in this case, $T$ plays the role of sample size $N_t$). 

\section{Computational experiments}

In this section, we test the proposed method in both synthetic and real datasets. We point out that the single-parameter version of the proposed algorithm was showcased in \cite{fattahi2021scalable}, thus in this section we seek to demonstrate the parametric version of the algorithm.

In \S\ref{sec:synt} we consider synthetically generated instances of time-varying GMRFs and compare the performance of our proposed method with two other well-known techniques. In \S\ref{sec:real}, we test our approach for time-varying DMRFs, using stock market data for our case study.

\subsection{Experiments with synthetic data}\label{sec:synt}

We consider randomly generated instances of sparsely-changing GMRFs. Given the dimension $n$, the number of time periods $T$ and the number of observations per time period $N_t$, the instances are constructed as follows. Initially, for time period $t=0$, we construct a ``true" precision matrix $\bar \Theta_t\in \mathbb{R}^{n\times n}$ with exactly $3n$ off-diagonal non-zero elements. Each non-zero off-diagonal element satisfies $(\bar\Theta_t)_{ij}=-0.4$, while $$(\bar\Theta_t)_{ii}=1+\sum_{j\neq i}|(\bar\Theta_t)_{ij}|$$ to guarantee positive definiteness. Then, for all subsequent time periods, 4\% of the non-zero off-diagonal elements are set to $0$, and the same number of previously zero elements are set to the value $-0.4$. Once the precision matrices have been constructed, we generate for each period $N_t$ iid samples $x~\sim \mathcal{N}(0, \bar\Theta_t^{-1})$ that are used to train the models, and another $N_t$ samples to use for validation. Thus, at the end of this process, we obtain two datasets, each containing $T\times N_t$ samples of dimension $n$. Finally, for each combination of parameters, we repeat this process five times and report the averages over the five instances generated with identical parameters.

\subsubsection{Methods} We test three approaches:\newline
$\bullet$ \textbf{ProxGL with cross-validation} We use the parametric algorithm described in \S\ref{sec_algorithm}, which solves \ref{generalopt1} for all values of the parameter $\gamma$. We then use cross-validation to determine the best value of $\gamma$ as follows. Given any fixed value of $\gamma$, the method produces a sequence of estimated precision matrices $\{\hat \Theta_t\}_{t=0}^T$. Then, given the sequence of samples in the validation set for time period $t$, $\{x_{t}^{(i)}\}_{i=1}^{N_t}$, the negative log-likelihood that the samples where generated from the estimated precision matrices is proportional to  $$-\frac{N_t}{2}\log (\text{det}(\hat \Theta_t))+\frac{1}{2}\sum_{i=1}^{N_t}\left(x_{t}^{(i)}\right)^\top \hat \Theta_t x_{t}^{(i)},$$
and the overall negative log-likelihood can be obtained by summing across all time periods. We select the value of $\gamma$ that minimizes the negative log-likelihood. In our computations, we set the $\ell_\infty$ parameter $\lambda=0.2$ ( this parameter naturally matches the data-generation process, where parameters change by multiples of $0.4$), and the shrinkage parameter $\nu=\nu_0\sqrt{(\log n)/ (T \cdot N_t)}$ for $\nu_0\in \{0.0,0.2,0.5,0.8,2.0\}$. Observe that we could select $\nu_0$ using cross-validation as well. Instead, we show that while the two extreme choices $\nu_0\in \{0,2\}$ (corresponding to no shrinkage, or excessive shrinkage) are unsatisfactory, all other choices of $\nu_0$ outperform the alternatives. 
\newline
$\bullet$ \textbf{TVGL with cross-validation} The time varying Graphical Lasso (TVGL), computed as~\eqref{mle_reg_gmrf}, is a well-known regularized MLE approach for estimating the sparsely-changing GMRFs~\citep{hallac2017network, cai2018capturing}. Similar to~\ref{generalopt1}, we pick $\gamma_1$ and $\gamma_2$ via cross-validation, i.e., by selecting the parameters that minimize the negative log-likelihood.
\newline
$\bullet$ \textbf{L1E with cross-validation} Consider an $\ell_1$ relaxation of our proposed estimator~\ref{generalopt1}, where the $\ell_0$ regularizer in the objective function is replaced with its $\ell_1$ relaxation. The resulted estimator reduces to that of~\cite{yang2014elementary} for $T=0$, and to that of~\cite{wang2018fast} for $T=1$ and $\gamma = 1$. Similar to ProxGL and TVGL, we fine-tune the parameter $\gamma$ via cross-validation. 

\subsubsection{Metrics} 

We test the quality of the approaches in terms of the estimation error with respect to the true parameters, as well as the ability to correctly recover the sparsity pattern of the sequence of precision matrices and their changes. The estimation error is computed as $$\texttt{error}=\sqrt{\frac{\sum_{t=0}^T\sum_{i=1}^n\sum_{j=i}^n \left(\Theta^\star_{t;ij}-\hat\Theta_{t;ij}\right)^2}{\sum_{t=0}^T\sum_{i=1}^n\sum_{j=i}^n  {\Theta^\star_{t;ij}}^2}},$$ 
where $\{\Theta^\star_t\}_{t=0}^T$ are the ``true" precision matrices of the process that generated the data, and $\{\hat \Theta_t\}_{t=0}^T$ are the estimates delivered by a given method. The ability to recover the true sparsity pattern is computed using the 
$$\texttt{F1-score}=2\times \frac{\texttt{Recall}\times \texttt{Precision}}{\texttt{Recall}+\texttt{Precision}},$$
where $\texttt{Recall}=\texttt{TP}/(\texttt{TP}+\texttt{FP})$, $\texttt{Precision}=\texttt{TP}/(\texttt{TP}+\texttt{FN})$, and $\texttt{TP}$, $\texttt{FP}$ and $\texttt{FN}$ denote respectively the number of true positives, false positives and false negatives in the sequence of estimated precision matrices. We also use the \texttt{F1-score} to evaluate the detection of changes in the sparsity patterns, where the formulas are identical but $\texttt{TP}$, $\texttt{FP}$ and $\texttt{FN}$ refer to changes in the sparsity pattern of the matrices from a time period to the next. 

\subsubsection{Performance} We test the performance of all methods by varying the number of samples per time period. Specifically, we fix $n=50$, $T=10$ and let $N_t= n\kappa$ for integer $1\leq\kappa\leq 20$. Table~\ref{tab:statsSynt} (at the end of the paper) presents the results for $\kappa=\{1,5,10,15,20\}$ and varying levels of threshold in the approximate backward mapping. Note that $\nu_0=0$ corresponds to no thresholding, whereas $\nu_0=2$ sets every off-diagonal entry of the sample covariance matrix to zero. As can be seen in the table, ProxGL with $\nu_0\in \{0,2\}$ is clearly inferior to other choices of the parameters, so these values are excluded from our subsequent experiments. Figure~\ref{fig:synt_metrics} depicts the performance of ProxGL, TVGL, and L1E for more values of $N_t/n$. We note that both ProxGL and TVGL tend to perform better as the number of samples $N_t$ increases (as expected). While TVGL results in a better estimation error with few samples ($N_t=2n$), ProxGL performs better for all values of $N_t\geq 8n$ and all parameters $\nu_0\in \{0.2,0.5,0.8\}$. In particular, for the largest value $N_t=40n$, ProxGL decreases the estimation error from 7.4\% of TVGL to up to 5.3\% (if $\nu_0=0.2$). We also observe that both TVGL and ProxGL have good performance in terms of detecting the true sparsity pattern of $\hat \Theta_t$.
However, TVGL fails to correctly identify changes in the underlying graphical model (with a \texttt{F1-score} of $0.2$ for most values of $N_t$), whereas ProxGL is able to do so accurately (with $\texttt{F1-score}$ above to $0.7$ for the largest values of $N_t$ considered). Method L1E performs poorly according to the three metrics considered, and is inferior to the other two approaches. Finally, we point out that no value of $\nu_0$ for ProxGL clearly dominates the others, with larger values resulting in better performance when $N_t$ is small and smaller values performing best with a large number of samples per time period.

We now briefly discuss computational times. TVGL is solved via the solver MOSEK~\citep{mosek}, which is one of the fastest available solvers for general-purpose convex optimization problems. We highlight the fact that while faster algorithms for static Graphical Lasso (such as QUIC~\citep{hsieh2014quic} and BIG-QUIC~\citep{hsieh2013big}) exist, these algorithms do not readily extend to the time-varying settings. Moreover, L1E reduces to linear programming (LP), for which fast algorithms (such as GUROBI~\citep{gurobi}, CPLEX~\citep{cplex2009v12}, and MOSEK) exist. However, these solvers are incapable of solving LPs parametrically, thereby leading to larger runtimes with cross-validation. In our experiments, we use MOSEK to solve L1E.

Our implementation of \ref{generalopt1} is very fast: for the large-sample instance with $N_t=40n$, it requires, on average, 11 seconds per instance on a single thread. We point out that this time includes computing the backward mappings, solving \ref{generalopt1} for all values of $\gamma$, and printing to disc (a CSV file) the estimated precision matrices for all possible values of $\gamma$. We also highlight that the most expensive component is, in fact, printing to file, which accounts for 9 out of the 11 seconds on average, while computing the backward mappings and solving \ref{generalopt1} is done in less than two seconds. On the other hand, for the same large-sample instance with $N_t=40n$,  TVGL with cross-validation over 16 choices of the parameters $(\gamma_1,\gamma_2)$ requires on average 1600 seconds per instance on a single thread, which is at least 145 times slower than \ref{generalopt1}. Finally, L1E with cross-validation over 10 choices of the parameter $\gamma$ requires on average 61 seconds per instance on a single thread, which is at least 5 times slower than \ref{generalopt1}.

\begin{figure}[!h]
	\centering
	\subfloat[\texttt{F1-score}: parameters]{\includegraphics[width=0.45\textwidth,trim={11cm 5.5cm 11cm 5.5cm},clip]{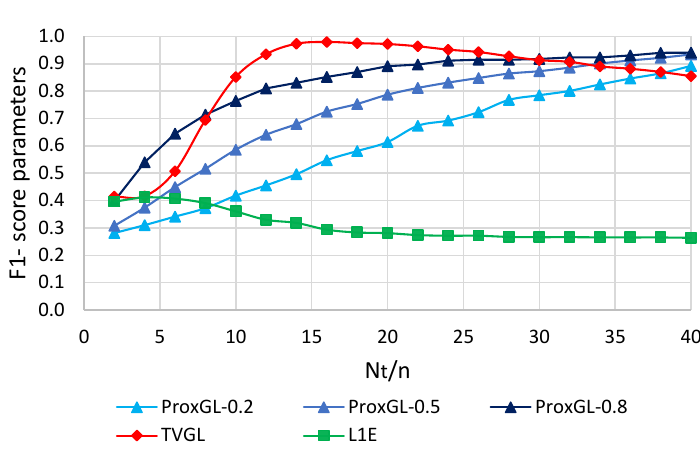}}\hfill\subfloat[\texttt{F1-score}: differences between parameters]{\includegraphics[width=0.45\textwidth,trim={11cm 5.5cm 11cm 5.5cm},clip]{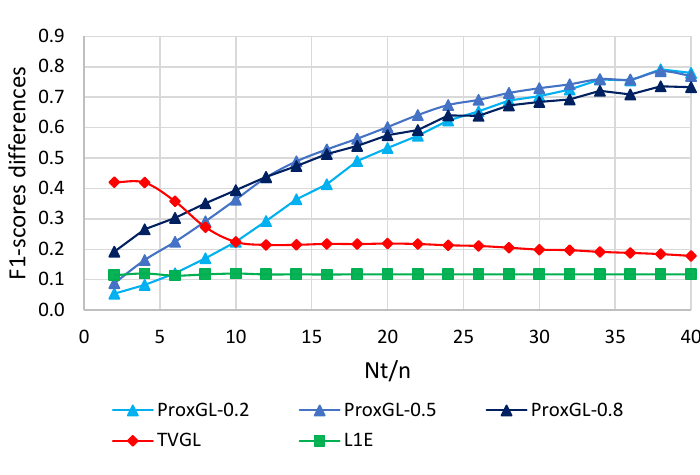}}\hfill	\subfloat[Estimation \texttt{error}]{\includegraphics[width=0.55\textwidth,trim={11cm 5.5cm 11cm 5.5cm},clip]{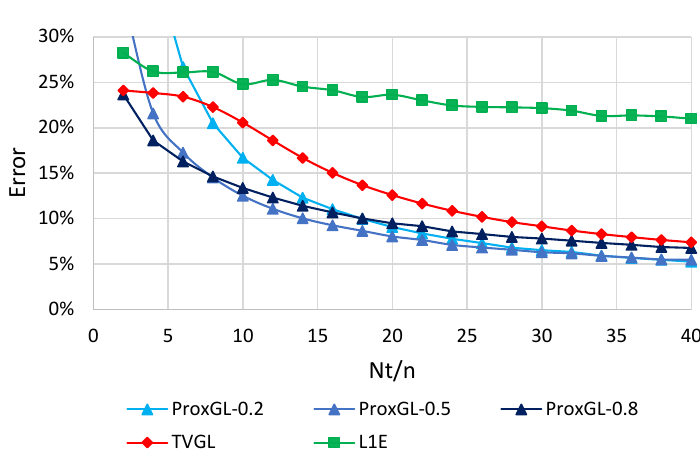}}
	\caption{\small Performance metrics between ProxGL, TVGL and L1E. The top row depicts the ability of each method to recover the sparsity pattern and identify changes in the underlying graph (top left: F1-score of the parameters, top right: F1-score of changes in the parameters), and bottom graphs depict the actual estimation error. }
	\label{fig:synt_metrics}
\end{figure}

In summary, we see that while TVGL has a reasonable performance and can identify the sparsity pattern of the precision matrices, ProxGL is able to decrease further the estimation error and is substantially better at identifying changes in the underlying graphs, and does so with only a fraction of the computational cost. 

\subsubsection{Further experiments on computational times}

We further discuss the computational times of solving \ref{generalopt1} for larger instances of GMRFs. To do so, we fix $T=10$ and $N_t=n/2$, and vary the dimension $n$ (and thus the number of parameters to be estimated). The computational times for DMRF and with varying $T$ are discussed in the next section. The results are summarized in Table~\ref{tab:timesSynt}. It shows the dimension $n$ and the total number of parameters to be estimated---computed as $n(n+1)T/2$. Moreover, it shows the time (in seconds) required to compute the backward mapping, as well as the time to solve the dynamic program for all values of $\gamma$, and the overall computational time. We see that in these instances, the computational times are primarily dominated by the backward mapping process, involving the inversion of $T$ matrices of size $n\times n$. Conversely, the time spent on solving \ref{generalopt1} itself is negligible.

\begin{table}[!h]
	\begin{center}
		\caption{Time required to compute estimators  \ref{generalopt1} (for all values of $\gamma$) on synthetic instances. The time to print the solutions to a file is not included. }
		\label{tab:timesSynt}
		\setlength{\tabcolsep}{2pt}
		\begin{tabular}{ c c |c c c}
			\hline
			\multirow{2}{*}{$\mathbf{n}$} & \multirow{2}{*}{\textbf{\# params}} & \multicolumn{3}{c}{\underline{\textbf{time (s)}}}\\
			&&\textbf{backwards mapping}\ &\ \textbf{dynamic program} \ &\ \textbf{total}\\
			\hline
			100& 50,500& 0.3 & 0.1 &0.4\\
			500& 1,252,500&4.9&0.9&5.8\\
			1,000&5,005,000&71.2&5.0&76.2\\
			1,500&11,257,500&250.2&11.2&261.4\\
			2,000&20,010,000&450.1&15.7&465.8\\
			\hline
		\end{tabular}
	\end{center}
\end{table}

\subsection{Case study with Stock Market data}\label{sec:real}

In this section, we test our approach for DMRFs on real stock market data. The dataset, based on the publicly available data from Kaggle\footnote{\url{https://www.kaggle.com/datasets/borismarjanovic/price-volume-data-for-all-us-stocks-etfs}}, contains the daily percent changes for 214 securities from 01/01/1990 to 08/10/2017 (note that only days in which trades occurred are reported, resulting in 7,022 time periods). To obtain a discrete dataset, we use an approach similar to \cite{campajola2022modelling}: if the absolute value of the percent change of a given security $i$ in a given time period $t$ is larger than some predetermined quantity $\kappa$, we set $x_{t;i}=1$ (indicating large volatility for that security at time period $t$), and otherwise we set $x_{t;i}=0$. The quantity $\kappa$ is set so that 50\% of the elements of $x$ are $0$. In total there are $91,591$ parameters to be estimated per time period, resulting in $91,591\times T$ parameters (e.g., if $N=30$, then $91,591\times T\approx 2.1\times10^7$).

In our experiments, given a number $N\in \{20,30,40,50,60\}$ of observations per time period, we partition the data into $T=\lfloor7,022/N\rfloor$ time periods (for simplicity, we discard the observations corresponding to the last time period with less than $N$ data points). We use the Kernel averaging method discussed in \S\ref{sec:kernelAvg} (using a Gaussian kernel with parameter $h=0.02 T^{-1/3}$), we set the $\ell_\infty$ parameter $\lambda_t=\lambda_0\sqrt{n/(TNh)}$ for $\lambda_0\in \{0.05,0.16,0.5,1,2,3,4,5\}$ and use the algorithm described in \S\ref{sec_algorithm} to compute the complete solution path with $q=0$.

Note that in this case we do not have access to an explicit ``ground truth", although periods of high volatility in the market tend to coincide with economic recessions. Observe that during the time period considered, there are three official recessions: the early 1990s recession (from July 1990 to March 1991), the early 2000s recession (from March to November 2001) and the Great Recession (from December 2007 to June 2009). Thus we expect to see more changes in the stock correlation network during those time periods. In a previous paper \citep{fattahi2021scalable}, we showed that this is indeed the case for selected choices of the hyperparameters of the model (although we were using GMRFs to model the problem). In this paper, since we can now compute the solution path for all values of $\gamma$, we investigate the complete distribution of changes in the stock correlation network for all optimal solutions of \ref{generalopt1}. 

In particular, fixing $N=30$ and given any choice of the $\ell_\infty$ parameter $\lambda$, we count the total number of changes in the stock correlation market at each time period across all optimal solutions of \ref{generalopt1} (for different values of $\gamma$), and report the resulting histograms in Figure~\ref{fig:real_histogram}. At the two extreme values of the parameter $\lambda$, i.e., $\lambda\to 0$ and $\lambda\to\infty$, the histograms are relatively flat: indeed, letting $\lambda\to 0$ induces changes for all parameters at all time periods, and letting $\lambda\to\infty$ results in estimations where all parameters are constant and equal to $0$. However, we see that for the vast majority of values of $\lambda$ (i.e., $0.16\leq \lambda\leq 5$), the periods with the most changes correspond to the early 2000s recession and the Great Recession. We thus conclude that the inferences obtained from \ref{generalopt1} are robust to the choices of sparsity parameter, as most of the solutions obtained coincide in identifying the two recessions (and, to a minor degree, the early 1990s recession as well). 
\begin{figure}[!h]
	\centering
	\subfloat[$\lambda=0.05$]{\includegraphics[width=0.33\textwidth,trim={11cm 5.5cm 11cm 5.5cm},clip]{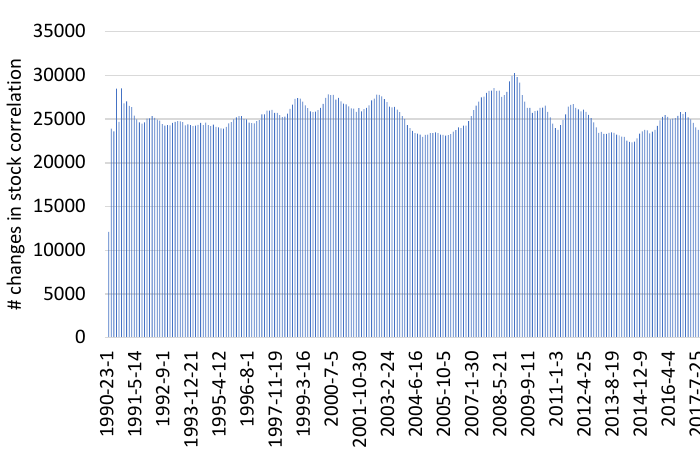}}\hfill	\subfloat[$\lambda=0.16$]{\includegraphics[width=0.33\textwidth,trim={11cm 5.5cm 11cm 5.5cm},clip]{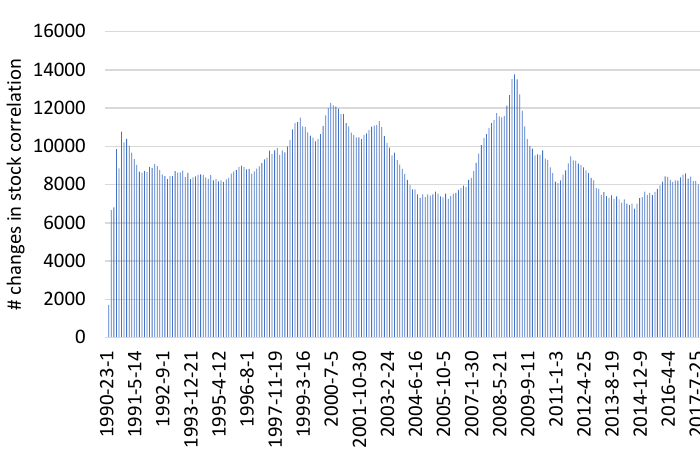}}\hfill	\subfloat[$\lambda=0.50$]{\includegraphics[width=0.33\textwidth,trim={11cm 5.5cm 11cm 5.5cm},clip]{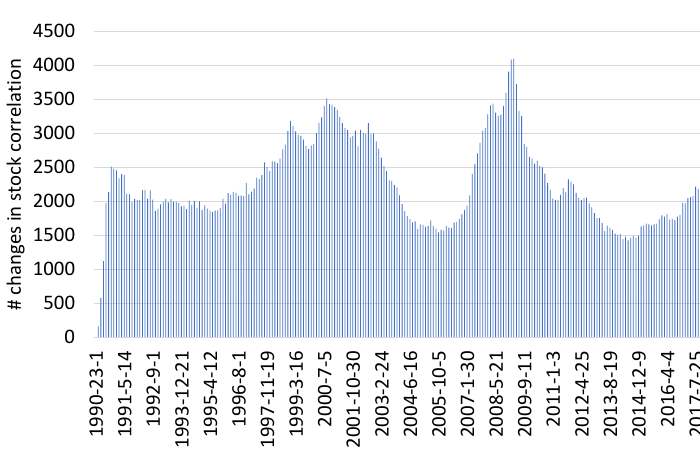}}\hfill	\subfloat[$\lambda=1.00$]{\includegraphics[width=0.33\textwidth,trim={11cm 5.5cm 11cm 5.5cm},clip]{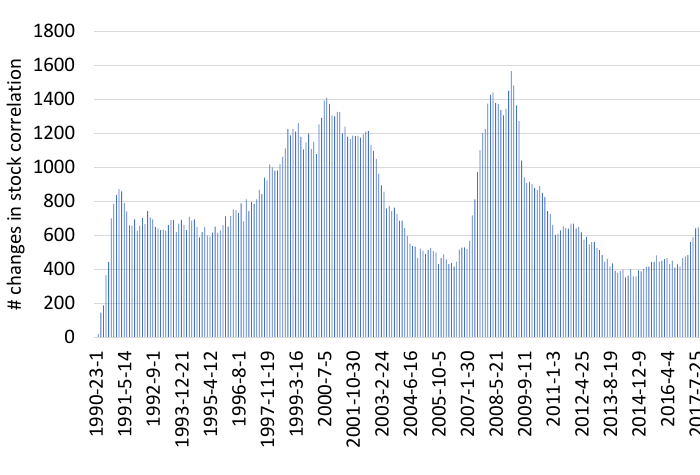}}\hfill	\subfloat[$\lambda=2.00$]{\includegraphics[width=0.33\textwidth,trim={11cm 5.5cm 11cm 5.5cm},clip]{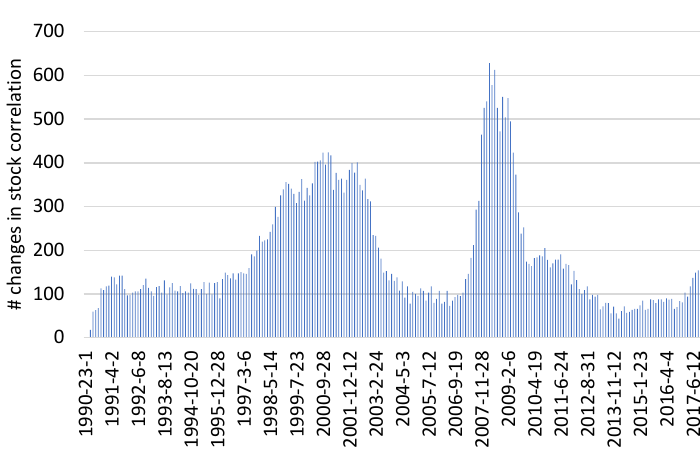}}\hfill	\subfloat[$\lambda=3.00$]{\includegraphics[width=0.33\textwidth,trim={11cm 5.5cm 11cm 5.5cm},clip]{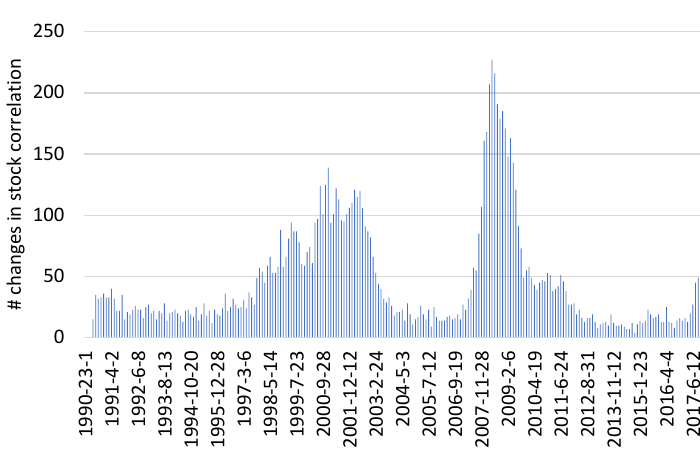}}
	\hfill	\subfloat[$\lambda=4.00$]{\includegraphics[width=0.33\textwidth,trim={11cm 5.5cm 11cm 5.5cm},clip]{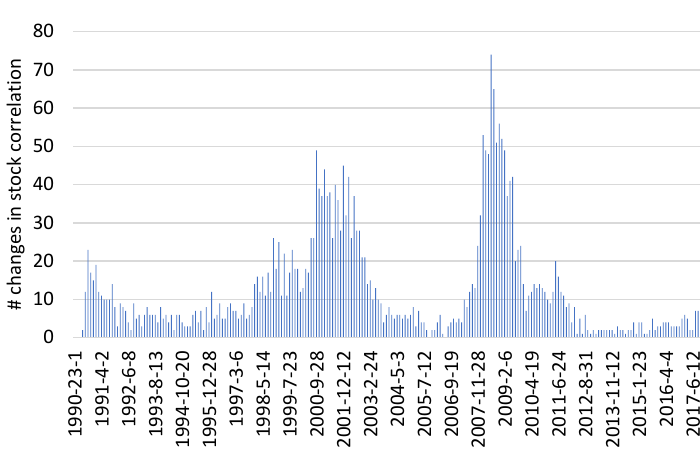}}\hspace{0.3cm}
	\subfloat[$\lambda=5.00$]{\includegraphics[width=0.33\textwidth,trim={11cm 5.5cm 11cm 5.5cm},clip]{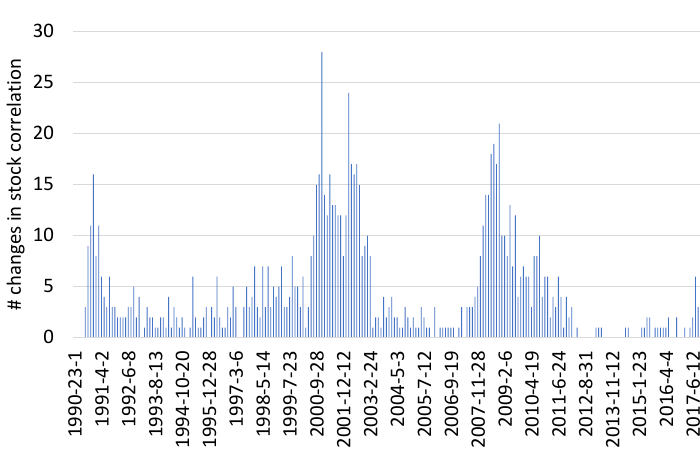}}\hfill
	\caption{\small Distribution of changes in the stock correlation network among all optimal solutions of \ref{generalopt1}.}
	\label{fig:real_histogram}
\end{figure}

Table~\ref{tab:timesReal} presents the computational times required to compute the solution path of \ref{generalopt1}. It shows the number of observations per time period $N$ and the total number of time periods $T$, and the total number of parameters to be estimated. Moreover, it shows the time required (in seconds) to compute the backwards mappings, solve the dynamic program, and the overall times. We point out some key differences with the computational times with GMRFs and synthetic instances reported in Table~\ref{tab:timesSynt}. First, computing the backward mappings for GMRFs is an expensive operation since requires matrix inversion, while the approximate backward mapping for DMRFs is substantially simpler to compute. As such, the time required to compute the backward mapping is negligible. Second, in the instances with real data, we vary the number of time periods $T$ (instead of the number of parameters per time period $p$). We see that the dynamic programming algorithm to solve \ref{generalopt1} is sensitive to $T$, and thus solving this problem consumes the majority of the computational time. Nonetheless, the computational times per number of parameters are similar to those reported in Table~\ref{tab:timesSynt}. Moreover,
in Figure~\ref{fig:loglog}, we plot in logarithmic scale the time required to solve the problems as a function of the number of time periods. We observe that the empirical complexity of $\mathcal{O}(T^{3.1})$ to compute the entire solution path matches the theoretical complexity reported in Theorem~\ref{thm_runtime}.
\begin{table}[!h]
	\begin{center}
		\caption{Time required to solve  \ref{generalopt1} (for all values of $\gamma$) on real instances with stock correlation networks. Time to print the solutions to a file is not included. }
		\label{tab:timesReal}
		\setlength{\tabcolsep}{2pt}
		\begin{tabular}{ c c c |c c c}
			\hline
			\multirow{2}{*}{$\mathbf{N}$} & \multirow{2}{*}{$\mathbf{T}$} &\multirow{2}{*}{\textbf{\# params}} & \multicolumn{3}{c}{\underline{\textbf{time (s)}}}\\
			&&&\textbf{backwards mapping}&\textbf{dynamic program} & \textbf{total}\\
			\hline
			60& 117& 10,716,147 & 8 &14&22\\
			50& 140&12,822,740&7&22&29\\
			40&176&16,120,016&8&57&65\\
			30&234&21,432,294&8&105&113\\
			20&351&32,148,441&9&694&701\\
			\hline
		\end{tabular}
	\end{center}
\end{table}  

\begin{figure}[!h]
	\centering
	\includegraphics[width=0.6\textwidth,trim={11cm 5.5cm 11cm 5.5cm},clip]{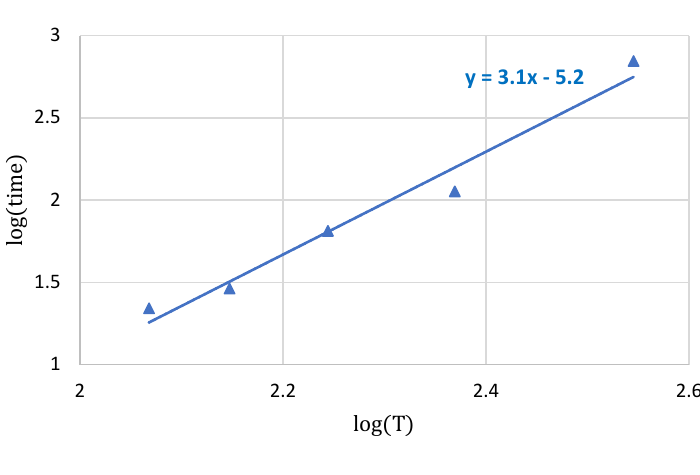}
	\caption{\small Time (in seconds) required in to solve stock correlation problems as a function of the number of time periods $T$ (in logarithmic scale). The slope of $3.1$ confirms the theoretical complexity of $\mathcal{O}(T^3)$ of the proposed method.}
	\label{fig:loglog}
\end{figure}

\section{Conclusion}
In this paper, we present a tractable formulation, referred to as \ref{generalopt1}, for inferring time-varying Markov random fields (MRFs) with various sparsity and temporal structures. Unlike existing methods that rely on maximum likelihood estimation (MLE) with relaxed $\ell_1$ regularization, we demonstrate that \ref{generalopt1} can efficiently achieve optimal solutions with exact $\ell_0$ regularization. By utilizing our proposed formulation, we can recover the complete solution path of the time-varying MRF for all sparsity parameters $\gamma$.

We identify two promising avenues for future research. Firstly, our parametric solution relies on a fixed $\lambda_t$ in \ref{generalopt1}. It would be valuable to explore whether \ref{generalopt1} can be parametrically solved for varying values of $\gamma$ and $\lambda$. Secondly, extending our proposed technique to spatiotemporal settings, where the MRF evolves over both time and space, presents an interesting future direction. Spatiotemporal MRFs are integral in inferring dynamically changing gene expression networks with spatial implications, carrying significant implications for understanding dynamic disease processes~\citep{werhli2006comparative, hartemink2000using, karlebach2008modelling}. The main challenge in this extension lies in the complex interactions among individual MRFs, which become considerably more intricate in spatially varying settings compared to their temporal counterparts.

\begin{sidewaystable}[!h]\small
	\begin{center}
	\caption{\texttt{F1-score} of precision matrices (F1\_p), differences (F1\_d), and estimation \texttt{error} for method ProxGL for different values of the shrinkage parameter $\nu$. }
 \label{tab:statsSynt}
	\begin{adjustbox}{width=\columnwidth,center}
		\setlength{\tabcolsep}{1pt}
		\begin{tabular}{ c |c c c | c c c | c c c | c c c | c c c| c c c | c c c}
			\hline
			\multirow{2}{*}{$\bm{\frac{N_t}{n}}$} &\multicolumn{3}{c|}{\underline{\textbf{ProxGL }$(\bm{\nu_0=0.0})$}}&\multicolumn{3}{c|}{\underline{\textbf{ProxGL }$(\bm{\nu_0=0.2})$}}&\multicolumn{3}{c|}{\underline{\textbf{ProxGL }$(\bm{\nu_0=0.5})$}}&\multicolumn{3}{c|}{\underline{\textbf{ProxGL }$(\bm{\nu_0=0.8})$}}&\multicolumn{3}{c|}{\underline{\textbf{ProxGL }$(\bm{\nu_0=2.0})$}}&\multicolumn{3}{c|}{\underline{\textbf{TVGL}}}&\multicolumn{3}{c}{\underline{\textbf{L1E}}}\\
   &\textbf{F1\_p}&\textbf{F1\_d}&\textbf{err}&\textbf{F1\_p}&\textbf{F1\_d}&\textbf{err}&\textbf{F1\_p}&\textbf{F1\_d}&\textbf{err}&\textbf{F1\_p}&\textbf{F1\_d}&\textbf{err}&\textbf{F1\_p}&\textbf{F1\_d}&\textbf{err}&\textbf{F1\_p}&\textbf{F1\_d}&\textbf{err}&\textbf{F1\_p}&\textbf{F1\_d}&\textbf{err}\\
   \hline
2&0.27&0.04&177.4\%&0.28&0.05&78.8\%&0.31&0.09&34.4\%&0.40&0.19&23.7\%&0.44&0.35&24.7\%&0.41&0.42&24.1\%&0.40&0.11&28.2\%\\
10&0.36&0.14&23.7\%&0.42&0.22&16.7\%&0.59&0.36&12.5\%&0.76&0.39&13.4\%&0.59&0.41&21.2\%&0.85&0.22&20.6\%&0.36&0.12&24.8\%\\
20&0.49&0.39&11.6\%&0.61&0.53&9.1\%&0.79&0.60&8.0\%&0.89&0.58&9.5\%&0.75&0.44&19.0\%&0.97&0.22&12.6\%&0.28&0.12&23.6\%\\
30&0.65&0.63&7.8\%&0.78&0.70&6.5\%&0.87&0.73&6.3\%&0.92&0.68&7.8\%&0.83&0.46&17.1\%&0.91&0.20&9.2\%&0.27&0.12&22.2\%\\
40&0.80&0.76&6.0\%&0.89&0.78&5.3\%&0.93&0.77&5.5\%&0.94&0.73&6.8\%&0.89&0.46&15.3\%&0.85&0.18&7.4\%&0.26&0.12&21.0\%\\
		\end{tabular}
  \end{adjustbox}
	\end{center}
\end{sidewaystable}  

\bibliography{reference}
\end{document}